\documentclass[12pt]{amsart}

\usepackage{amsmath}
\usepackage{amsthm}
\usepackage{amssymb}
\usepackage{enumitem}

\usepackage{color}
\usepackage[usenames,dvipsnames]{xcolor}

\usepackage{xspace}

\usepackage[T1]{fontenc}
\usepackage{times}
\usepackage{hyperref}

\usepackage{todonotes}

\author {Beno\^\i{}t Collins}
\address{D\'epartement de Math\'ematique et Statistique, Universit\'e d'Ottawa,
585 King Edward, Ottawa, ON, K1N6N5 Canada, \newline \indent
WPI AIMR, Tohoku, Sendai, 980-8577 Japan
\newline \indent 
CNRS, Institut Camille Jordan Universit\'e  Lyon 1,
France}
\email{bcollins@uottawa.ca}

\author{Piotr \'Sniady}
\address{Zentrum Mathematik, M5,
Technische Universit\"at M\"unchen, \linebreak
Boltzmannstrasse 3,
85748 Garching, Germany \newline \indent
Instytut Matematyczny, Polska Akademia Nauk, \linebreak
\mbox{ul.~\'Sniadec\-kich 8,} 00-956 Warszawa, Poland
 \newline
\indent 
Instytut Matematyczny,
Uniwersytet Wroc\l{}awski,  \mbox{pl.~Grunwaldzki~2/4,} 50-384
Wroclaw, Poland
} 
\email{piotr.sniady@tum.de, piotr.sniady@math.uni.wroc.pl}

\title[Asymptotic fluctuations of representations of unitary
groups]{Asymptotic fluctuations of \\ representations of the unitary groups}



\theoremstyle{plain}
\newtheorem{lemma}{Lemma}[section]
\newtheorem{theorem}[lemma]{Theorem}
\newtheorem{proposition}[lemma]{Proposition}
\newtheorem{corollary}[lemma]{Corollary}

\theoremstyle{definition}

\theoremstyle{remark}
\newtheorem{remark}[lemma]{Remark}

\newtheorem{problem}[lemma]{Problem}

\newcommand{\A}{{\mathfrak{A}}}

\newcommand{\E}{{\mathbb{E}}}
\newcommand{\C}{{\mathbb{C}}}
\newcommand{\R}{{\mathbb{R}}}
\newcommand{\Z}{{\mathbb{Z}}}
\newcommand{\g}{{\mathfrak{g}}}

\newcommand{\uu}{{\mathfrak{u}}}
\newcommand{\su}{{\mathfrak{su}}}
\newcommand{\gl}{{\mathfrak{gl}}}

\newcommand{\U}{{\mathfrak{U}}}

\newcommand{\V}{\mathcal{V}}
\newcommand{\W}{\mathcal{W}}

\newcommand{\El}{{\mathcal{L}}}
\newcommand{\M}{{\mathbb{M}}}

\newcommand{\gwia}{^{\star}}

\newcommand{\Sy}[1]{\mathfrak{S}({#1})}

\newcommand{\momentNumber}{\mathfrak{m}}
\newcommand{\momentTensor}{\mathfrak{M}}

\newcommand{\len}{r}
\newcommand{\rank}{r}
\newcommand{\rankB}{s}

\newcommand{\spectralNatural}{{\mu}}
\newcommand{\spectralNaiveShifted}{{\widehat{\mu}}}

\newcommand{\naiveMatrix}{X}
\newcommand{\naturalMatrix}{Y}
\newcommand{\genericMatrix}{Z}

\DeclareMathOperator{\trace}{tr} 

\newcommand{\trV}{\trace_V}
\newcommand{\trn}{\trace_n}
\newcommand{\trVn}{\trace_{V_n}}

\DeclareMathOperator{\Tr}{Tr}

\DeclareMathOperator{\Var}{Var}
\DeclareMathOperator{\End}{End} 
\DeclareMathOperator{\dimm}{dim}
\DeclareMathOperator{\Id}{Id}

\DeclareMathOperator{\Ad}{Ad}

\DeclareMathOperator{\Cov}{Cov}
\DeclareMathOperator{\degg}{deg}

\newcommand{\naive}{na\"{\i}ve\xspace}

\begin{document}

\begin{abstract}
We study asymptotics of representations of the unitary groups $U(n)$
in the limit as $n$ tends to infinity and we show that in many aspects they behave like
large random matrices. In particular, we prove that the highest weight of a random
irreducible component in the Kronecker tensor product of two irreducible
representations behaves asymptotically in the same way as the spectrum of the
sum of two large random matrices with prescribed eigenvalues. This agreement
happens not only on the level of the mean values (and thus can be described
within Voiculescu's free probability theory) but also on the level of
fluctuations (and thus can be described within the framework of higher order
free probability).
\end{abstract}

\maketitle

\section{Introduction}
\label{sec:introduction}

\subsection{Asymptotics of representations of the unitary groups} 
\label{subsec:intro-intro}
In general, questions concerning representations of the unitary groups
$U(n)$ and manipulations with them, such as the
problem of decomposing the Kronecker tensor product of two irreducible
representations into a sum of irreducible components, have a well-known answer
given by algorithms involving some combinatorial objects, such as Young
tableaux \cite{Fulton1997}, weights \cite{FultonHarris,Brocker-tomDieck} or Littelmann paths \cite{Littelmann95}. 
However, in the limit $n\to\infty$, due to the computational complexity of such algorithms, is is very difficult to obtain relevant information about representations.
It is therefore natural to ask for some partial or approximate answers which would be useful and meaningful asymptotically. For similar problems in relation to the symmetric groups, we refer to the work of Biane
\cite{Biane1998}.

The first result in this direction is due to Biane \cite{Biane95}. He proved
that a typical irreducible component of a representation of the unitary group
$U(n)$ resulting from some natural representation-theoretic operations can be
asymptotically described in the language of
\emph{Voiculescu's free probability theory} \cite{VoiculescuDykemaNica}. This highly
non-com\-mu\-ta\-tive probability theory was known to describe the asymptotic
behavior of \emph{large random matrices} \cite{Voiculescu1991}.

In this paper, we revisit the work of Biane \cite{Biane95},
and give a conceptual explanation of the fact that
both representations and
random matrices are asymptotically described by Voiculescu's free probability.
Namely, we show that 
\emph{representations behave
asymptotically in the same way as large random matrices}. This equality of asymptotics 
concerns not only the
mean value (as in the original work of Biane \cite{Biane95}) but also 
fluctuations around the mean values. 
Our results are naturally expressed within the context of \emph{higher order free probability} \cite{HigherOrderFreeness1,HigherOrderFreeness2,HigherOrderFreeness3}
which was developed as a framework capable of describing fluctuations of random
matrices in an abstract manner.
Our above mentioned results reduce the original problem of the asymptotics of
representations of the unitary groups to the better and more widely understood
problem of large random matrices spectra. 

We also show that the technical assumption from the original paper of Biane
\cite{Biane95} concerning the speed of growth of a typical highest weight can be
significantly weakened.

The main method of proof is to associate to a representation of the unitary group 
$U(n)$ a certain \emph{$n\times n$ random matrix with non-commutative entries} and to show that under
some mild assumptions, this non-commutativity asymptotically tends to zero.
Hence, for $n\to\infty$ it can be regarded as a 
classical random matrix. 
A very similar approach was used in our
previous paper \cite{Collins'Sniady2006} in order to study asymptotics of
representations of a fixed compact Lie group.

In the remaining part of this section we introduce the basic notations and
present in more detail the main results of the paper.

\subsection{Representations and shifted weights for the unitary group}
We will use only some basic facts about Lie groups, Lie algebras and their representations. The books \cite{FultonHarris,Brocker-tomDieck} are good references to this topic. 
All representations considered in this paper are assumed to be finite-dimensional.

Any irreducible representation of the unitary group $U(n)$ is uniquely determined 
up to  equivalence 
by its \emph{highest weight} $\lambda$, 
which can be identified with a vector $\lambda=(\lambda_1\geq \cdots\geq \lambda_n)\in\Z^n$.
We define the \emph{shifted highest weight} $l=(l_1>\cdots>l_n)\in\Z^n$ by 
\[ l_i:=\lambda_i+(n-i).\]
For the purposes of this article it is more convenient to index irreducible representations
by their shifted highest weights instead of the usual highest weights;
for this reason we use the symbol $\rho_l$ to denote the corresponding
irreducible representation. 

The representation $\rho_l$ of the Lie group $U(n)$ gives rise (by differentiating in the identity) 
to a representation of the corresponding \emph{Lie algebra $\uu(n)$ of antihermitian matrices}.
We denote this representation by the same symbol  $\rho_l$. 
Since the Lie algebra $\uu(n)$ is not semisimple, it has irreducible representations other than $\rho_l$ over  $l=(l_1>\cdots>l_n)\in\Z^n$. 
However, since any representation of the Lie algebra $\uu(n)$ which will be considered in this paper corresponds to some 
representation of the Lie group $U(n)$, this will not create any difficulties.
Alternatively, one could consider rather the group $SU(n)$ and the corresponding semisimple Lie algebra $\su(n)$.

\subsection{The \naive random matrix associated to a representation}

To an irreducible representation $\rho=\rho_l$ of $U(n)$ (or, to an irreducible representation $\rho=\rho_l$ of Lie algebra $\uu(n)$) we associate a random matrix  
\begin{equation}
\label{eq:random-matrix}
\naiveMatrix=\naiveMatrix(\rho):=U \begin{bmatrix} l_1 &  &  \\ & \ddots & \\ & 
& l_n
\end{bmatrix} U^{-1}, 
\end{equation}
where $U$ is a random unitary matrix, distributed according to the Haar
measure on $U(n)$. Another way of defining this random matrix is to say that its
distribution is the uniform measure on the manifold of all hermitian matrices
with the eigenvalues specified by the shifted weight $l$.
We will call $X(\rho)$ the \emph{\naive random matrix associated to $\rho$}.
The terminology `\naive' here is introduced in order to distinguish this random matrix from the one which will be introduced in Section \ref{subsec:random-matrix-natural}.

If a representation $\rho$ is reducible, we consider its decomposition into
irreducible components
\[\rho=\bigoplus_{l\in\Z^l} n_l \cdot \rho_l,\]
where $n_l\in\{0,1,\dots\}$ denotes the multiplicity and we consider a
probability measure on the set of all shifted weights given as follows:
\begin{equation}
\label{eq:probability-on-weights}
 P(l):= \frac{n_l\cdot \text{(dimension of
$\rho_l$)}}{\text{(dimension of $\rho$)}}. 
\end{equation}
To such a reducible representation $\rho$, we associate a random matrix 
$\naiveMatrix(\rho)$ given by \eqref{eq:random-matrix}, where---as before---$U$ is a random unitary matrix distributed 
according to the Haar measure on $U(n)$, but $l$ should be now an independent random variable with the 
distribution given by \eqref{eq:probability-on-weights}.

The \naive random matrix $\naiveMatrix=\naiveMatrix(\rho)$ contains all information (up to ampliation) about the 
decomposition of the representation $\rho$ into irreducible components. 
In our previous paper \cite{Collins'Sniady2006}, we
studied applications of this matrix in the study of the asymptotics of
representations of a fixed unitary group $U(n)$ (and
of any fixed compact Lie group). In this article we focus on asymptotics of
representations $\rho_n$ of the unitary groups $U(n)$ in the limit $n\to\infty$.
Therefore, we will have to replace the random matrix $\naiveMatrix$ by a sequence of random
matrices $\big(\naiveMatrix(\rho_n) \big)$ with their sizes tending to infinity.

\subsection{The canonical random matrix with non-commutative entries associated to a representation}
\label{subsec:random-matrix-natural}

Let $\rho:\uu(n)\rightarrow\End(V)$ be a representation of the Lie algebra $\uu(n)$ (in the case when $\rho$ is a 
representation of Lie group $U(n)$ we replace $\rho$ by the corresponding representation of the Lie algebra). 
We associate to $\rho$ the following matrix
\[ \naturalMatrix(\rho) := 
\begin{bmatrix}
\rho(e_{11}) & \dots  & \rho(e_{1n}) \\
\vdots       & \ddots & \vdots       \\
\rho(e_{n1}) & \dots  & \rho(e_{nn}) 
\end{bmatrix} \in \M_n(\C) \otimes \End(V),\]
where $e_{ij}\in \M_n(\C) = \uu(n) \otimes_\R \C$ are the matrix units. 
We say that $\naturalMatrix(\rho)$ is the \emph{natural random matrix (with non-commutative entries) associated to $\rho$} 
(we postpone the exact definition of \emph{non-commutative random variables} to Section \ref{subsec:noncommutative-probability}).
We will discuss some fine details of this construction in Section \ref{subsec:representation-as-random-matrix}. 

This matrix plays a crucial role in our approach; it was first introduced by Biane \cite{Biane1998} in the context of the 
representation theory of the symmetric groups, see also the work of Kuperberg \cite{Kuperberg2002}.

\subsection{The main result}

The main result of this paper can be stated as follows:
\begin{theorem}
\label{theo:main_theorem}
Let $(\varepsilon_n)$ be a sequence of real numbers such that $\varepsilon_n=
o\left(\frac{1}{n}\right)$.
For each $n$, let $\rho_n$ be a representation of the unitary group $U(n)$. 

Then, the corresponding sequence of rescaled \emph{natural} random
matrices  $\big(\varepsilon_n \naturalMatrix(\rho_n)\big)$ converges in distribution  if and only if the sequence of rescaled 
\emph{\naive} random matrices
$\big(\varepsilon_n \naiveMatrix(\rho_n)\big)$ converges in distribution. 
In both cases the convergence is to be understood in the sense of
\emph{higher order free probability}. If the limits exist, they are equal. 
\end{theorem}

In particular, this theorem means that we can connect the `\naive' random matrix associated to a representation with its `natural' counterpart, and this provides a conceptual framework in which one can explain the similar behavior of representations and random matrices in the limit of large dimension.

The above theorem is stated more precisely and proved
 as Theorem \ref{theo:main_theorem-RELOADED}, after appropriate
notation is introduced.
In Section \ref{subsec:concrete-applications} of this introduction, we show some concrete applications of this abstract result to representation theory.

\subsection{Spectral measure for representations and random matrices}
\label{subsec:specral-measure-definition}
Let $\genericMatrix$ be an $n\times n$ hermitian random matrix, and
$l=(l_1\geq \cdots\geq l_n)\in\R^n$ the set of its
eigenvalues (counted with multiplicities); since $\genericMatrix$ is random,  
the vector $l$ of its eigenvalues is also random. We define the \emph{spectral
measure} of $\genericMatrix$ 
as the random probability measure on the real line
\begin{equation}
\label{eq:definition-spectral-measure}
\mu_\genericMatrix:= \frac{1}{n} \sum_{i} \delta_{l_i}.
\end{equation}

For an irreducible representation $\rho=\rho_l$ of $U(n)$ corresponding to
the shifted highest weight $l=(l_1> \cdots > l_n)\in\Z^n$ (or, for an irreducible representation $\rho=\rho_l$ of 
the Lie algebra $\uu(n)$), we define its
\emph{\naive spectral measure}
\begin{equation}
\label{eq:spectral-measure-duh}
\spectralNaiveShifted_\rho= \spectralNaiveShifted_l:= \frac{1}{n} \sum_{i}
\delta_{l_i} 
\end{equation}
which is a deterministic probability measure on $\R$. 

If $\rho$ is a reducible representation, we define its \naive spectral measure
$\spectralNaiveShifted_\rho$ by the
same formula \eqref{eq:spectral-measure-duh}, however now $l$ should be understood as  a \emph{random}
shifted highest weight as defined by \eqref{eq:probability-on-weights}. In this
case $\spectralNaiveShifted_\rho$ becomes a \emph{random} probability measure on $\R$.

The \naive spectral measure $\spectralNaiveShifted_\rho$
is nothing else but the spectral measure of the \naive random matrix $\naiveMatrix(\rho)$ associated to $\rho$.

If $\mu$ is a probability measure on $\R$ and $\varepsilon$ is a real number, we
 denote by $D_{\varepsilon}\mu$ the \emph{dilation} of the measure $\mu$.
It is the distribution of the random variable $\varepsilon Z$, where $Z$
is a random variable with the distribution $\mu$. We use the notational
shorthands
\begin{align*}
 \varepsilon l  & := (\varepsilon l_1,\dots,\varepsilon_n l_n) \qquad 
\text{for } l=(l_1,\dots,l_n), \\
 \spectralNaiveShifted_{\varepsilon \rho} & := D_{\varepsilon}
\spectralNaiveShifted_{\rho}, \\
 \naiveMatrix(\varepsilon \rho) & := \varepsilon \naiveMatrix(\rho).
\end{align*}

\subsection{Gaussian fluctuations of measures}
\label{subsec:gassian-fluctuations-measures}
Let $(\mu_n)$ be a sequence of random probability measures on $\R$. 
We will say that the \emph{fluctuations of $(\mu_n)$ are asymptotically
Gaussian (with covariance decay $\frac{1}{n^2}$)} if the limit 
\begin{equation}  
\label{eq:limit-mean}
\lim_{n\to\infty} \E \int_\R z^\rank \ d\mu_n 
\end{equation}
exists for any integer $\rank\geq 1$ and the joint distribution of the family of centered random variables 
\begin{equation}
\label{eq:limit-fluctuation}
\Bigg\{ n \left( \int_\R z^\rank \ d\mu_n - \E \int_\R z^\rank \ d\mu_n \right)  \Bigg\}_{r=1,2,3,\dots} 
\end{equation}
converges in moments to some Gaussian distribution (in the sense that the
distribution of any finite family converges). 

We say that two such sequences $(\mu_n)$, $(\mu_n')$ of random probability measures have
\emph{asymptotically the same Gaussian fluctuations (with covariance decay
$\frac{1}{n^2}$)} if they are asymptotically Gaussian in the above sense, their corresponding
limits \eqref{eq:limit-mean} are equal and the fluctuations
\eqref{eq:limit-fluctuation} converge to the same Gaussian limit.

\subsection{Applications of the main result}
\label{subsec:concrete-applications}

Let us now present a few concrete consequences of Theorem \ref{theo:main_theorem}. 
A more complete collection of its applications, together with the proofs, is given in 
 Section \ref{sec:applications}.

\subsubsection{Kronecker tensor product}
We start with a solution to the problem mentioned in the beginning of
Section \ref{subsec:intro-intro}, namely the decomposition of Kronecker tensor
products into irreducible components.

We recall that if $\rho_1:U(n)\rightarrow \End(V_1)$ and $\rho_2:U(n)\rightarrow \End (V_2)$ are representations of the same unitary group $U(n)$, their \emph{Kronecker tensor product}
$\rho_1\otimes \rho_2:U(n)\rightarrow \End(V_1 \otimes V_2)$ is a representation of $U(n)$ defined by diagonal action
\[ (\rho_1\otimes \rho_2)(u) := \rho_1(u) \otimes \rho_2(u). \]

\begin{corollary}
\label{corollary:Kronecker-free}
Let $(\varepsilon_n)$ be a sequence of real numbers such that $\varepsilon_n=
o\left(\frac{1}{n}\right)$. For each $i\in\{1,2\}$ and $n\geq 1$
let $\rho^{(i)}_n$
be an irreducible representation of $U(n)$.
Assume that for each $i\in\{1,2\}$ the sequence
$\left(\spectralNaiveShifted_{\varepsilon_n \rho^{(i)}_n}\right)_{n=1,2,\dots}$ of the
(rescaled) \naive spectral measures converges in moments to some probability measure
$\mu^{(i)}$.

Then the (rescaled) \naive spectral measure $\spectralNaiveShifted_{\varepsilon_n
\left( \rho^{(1)}_n \otimes \rho^{(2)}_n\right)}$ of the Kronecker tensor product 
converges in moments almost surely to \emph{Voiculescu's free
convolution} $\mu^{(1)} \boxplus \mu^{(2)}$.
\end{corollary}

Note that the almost sure convergence relies here on the fact that all random variables are defined on the same probability space.

A similar result was proved by Biane \cite{Biane95} under much stronger
assumptions on decay of $\varepsilon$, namely that
$\varepsilon=o\left(\frac{1}{n^\alpha}\right)$ for all values of the
exponent $\alpha$. 
 
Corollary \ref{corollary:Kronecker-free} is
formulated in terms of \emph{free additive convolution} which belongs to the language
of \emph{Voiculescu's free probability theory} \cite{VoiculescuDykemaNica}.
It can be strengthened by establishing a direct bridge with the theory
of unitarily invariant random matrices as in the following corollary.

\begin{corollary}
\label{corollary:Kronecker-RM}
Let the assumptions of Corollary \ref{corollary:Kronecker-free} be fulfilled.
For $i\in\{1,2\}$, we denote by $\naiveMatrix^{(i)}_n=\naiveMatrix(\varepsilon_n \rho_n^{(i)})$ the (rescaled) \naive
random matrix corresponding to the representation $\rho_n^{(i)}$.
Random matrices $\naiveMatrix^{(1)}_n$ and  $\naiveMatrix^{(2)}_n$ are chosen to be independent.

Then the (rescaled) \naive spectral measures $\left(\spectralNaiveShifted_{\varepsilon_n \left( \rho^{(1)}_n \otimes \rho^{(2)}_n \right)} \right)_{n=1,2,\dots}$ of  Kronecker tensor products
and the spectral measures of random matrices
$\left(  \naiveMatrix^{(1)}_n+\naiveMatrix^{(2)}_n \right)_{n=1,2,\dots}$ have asymptotically the same Gaussian fluctuations with
covariance decay $\frac{1}{n^2}$.
\end{corollary}

In the light of Corollary \ref{corollary:Kronecker-RM}, the contents of
Corollary \ref{corollary:Kronecker-free} should not come as a surprise, since it
is well known \cite{Voiculescu1991} that Voiculescu's free convolution
describes asymptotics of the spectrum of sum of two independent random
matrices.

\subsubsection{Restriction to a subgroup}
Similarly, we can handle the problem of restriction to a unitary subgroup.
In the following we consider the sequence of embeddings of the unitary groups $U(1)\subset U(2) \subset \cdots$
given by the natural map $U(n)\ni U \mapsto \begin{bmatrix} U & 0 \\ 0 & 1 \end{bmatrix} \in U(n+1)$.

\begin{corollary}
\label{coro:restriction}
Let $(\varepsilon_n)$ be a sequence of real numbers such that
$\varepsilon_n=o\left(\frac{1}{n}\right)$, for each $n\geq 1$ let $\rho_n$ be
an irreducible representation of $U(n)$ such that the sequence of (rescaled)
\naive spectral measures $\spectralNaiveShifted_{\varepsilon_n \rho_n}$ converges in
moments to some probability measure $\mu$. Let $(m_n)$ be a sequence of integers such that
$1\leq m_n\leq n$ and such that the limit $\alpha:=\lim_{n\to\infty}
\frac{m_n}{n}>0$ exists and is positive. For each $n$ we define $\rho'_n:=
\rho_n
\big\downarrow^{U(n)}_{U(m_n)}$ to be a representation of $U(m_n)$ given by the
restriction of $\rho_n$ to the subgroup.

Then, the sequence of (rescaled) \naive spectral measures
$\spectralNaiveShifted_{\varepsilon_n \rho'_n}$ converges almost surely in
moments to the \emph{free compression} of $\mu$ by a \emph{free projector of trace $\alpha$}
 (see \cite{VoiculescuDykemaNica} for a definition).

In addition, the (rescaled) \naive spectral measure $\spectralNaiveShifted_{\varepsilon_n\rho'_n}$ of the restricted representation
and the spectral measure of the $m_n\times
m_n$ upper-left corner of the random matrix $\naiveMatrix_n$ have asymptotically 
the same Gaussian fluctuations with covariance decay $\frac{1}{n^2}$.
\end{corollary}

\begin{problem}
What happens in the above Corollary \ref{coro:restriction} in the case when $\lim m_n=\infty$ but $\lim \frac{m_n}{n}=0$? We conjecture that the limiting distribution in this case is the semicircular law and instead of the corners of the matrices one can take some multiple of the (traceless?) GUE random matrix.
\end{problem}

\begin{remark}
Corollaries \ref{corollary:Kronecker-free}, \ref{corollary:Kronecker-RM} and \ref{coro:restriction} remain true if 
the \naive spectral measures $\spectralNaiveShifted_{\varepsilon \rho}$ of representations are replaced by the natural 
spectral measures $\spectralNatural_{\varepsilon \rho}$ which will be introduced in Section \ref{subsec:spectral-measure}.
\end{remark}

\begin{problem}
Do Corollary \ref{corollary:Kronecker-free}, Corollary \ref{corollary:Kronecker-RM} and 
Corollary \ref{coro:restriction} still hold true if we replace the convergence in moments by weak
convergence of probability measures? 
\end{problem}

\subsection{Elements of proof}
\label{subsec:elements-of-proof-II}

\subsubsection{Representations as non-commutative vectors}
In our previous paper \cite{Collins'Sniady2006}, we studied the asymptotics of a
sequence $(\rho_n)$ of representations of a fixed compact Lie
group $G$. The first main idea was that instead of the representation 
$\rho:G\rightarrow\End(V)$ of a Lie group, it is more convenient to consider its
derivative $\rho:\g\rightarrow\End(V)$ which is a representation of the corresponding Lie algebra $\g$.

The second main idea was that each representation
$\rho:\g\rightarrow\End(V)$ of the Lie algebra $\g$ can be
equivalently viewed as $\rho\in\g\gwia \otimes \End(V)$, where $\g\gwia$ denotes the
vector space dual to $\g$. Since $\End(V)$
equipped with the \emph{normalized trace} $\trV$ can be viewed as a
\emph{non-com\-mu\-ta\-tive probability space}, $\rho\in\g\gwia \otimes \End(V)$ becomes a
\emph{non-com\-mu\-ta\-tive random vector} in $\g\gwia$.
Our problem is therefore reduced to studying the sequence $(\varepsilon_n 
\rho_n)$
of non-commutative random vectors in $\g\gwia$, where $(\varepsilon_n)$ is some
suitably chosen sequence of numbers which takes care of the right normalization.
We proved that in many situations the distribution of $\varepsilon_n \rho_n$
converges to a classical (commutative) probability distribution on $\g\gwia$
which, when the group $G$ has some matrix structure, can be interpreted as some
random matrix. 

In this way, several problems of the asymptotic representation
theory of Lie group $G$ have answers in terms of certain random matrices and
their eigenvalues.

\subsubsection{The difficulty: fixed group replaced by a sequence of groups}
In the current paper, the fixed group $G$ is replaced by a sequence of 
groups $G_1,G_2,\dots$ (in fact, we concentrate on a very special case when 
$G_n=U(n)$ is the unitary group) and we study the asymptotic properties of the sequence
$(\rho_n)$, where $\rho_n$ is a representation of $G_n$. 
Our previous paper \cite{Collins'Sniady2006} is not directly
applicable because each $\rho_n$ is a non-commutative random vector in a
different space, namely $\g_n\gwia$ (where $\g_n$ is the Lie algebra of $G_n$), 
and it is not possible to consider the
limit of the distributions.
In the following, we show how to overcome this
difficulty and how to find a substitute for the notion of convergence in
distribution which will allow us to speak about asymptotic distribution of a
sequence of representations.

In Lemma \ref{lem:invariance-of-moments}, we prove that the
$\rank$-th \emph{moment of the representation} $\rho$ of Lie algebra $\g$ of a Lie group $G$
\[ \momentTensor_\rank(\rho)=\E \left(\rho^{\widehat{\otimes} \rank}\right) \in
(\g\gwia)^{\otimes \rank} \]
 is
invariant under the coadjoint action of $G$
(for the exact definition of the \emph{moment of the representation}, see Section \ref{subsect:random_vectors}). 
The set of such $G$-invariant
elements of $(\g\gwia)^{\otimes \rank}$ is denoted by $\left[(\g\gwia)^{\otimes
\rank}\right]_G$. For many groups $G$, the corresponding invariant spaces
$\left[(\g\gwia)^{\otimes \rank}\right]_G$ are surprisingly nice.

The common structure of the groups $(G_n)$ which turns out to be sufficient for
our purposes is the following one: we assume that for each $\rank$, the spaces
$\left[(\g_n\gwia)^{\otimes \rank}\right]_{G_n}$ 
are all isomorphic in some  
canonical way,
except possibly for finitely many values of $n$,
 to (a subspace
of) some abstract vector space denoted by $\left[(\g\gwia)^{\otimes \rank}\right]_{G}$; in
this way, we can regard the inclusions as follows:
\[ \momentTensor_\rank(\rho_n) \in \left[(\g_n\gwia)^{\otimes \rank}\right]_{G_n}
\subseteq
\left[(\g\gwia)^{\otimes \rank}\right]_G. \]
For each value of $\rank$ we choose some
basis in the invariant space $\left[(\g\gwia)^{\otimes
\rank}\right]_{G}$. Now, it makes sense to speak about the asymptotic behavior of the
coordinates of $\momentTensor_\rank(\varepsilon_n \rho_n)$ in this basis, for some
suitably chosen
sequence $(\varepsilon_n)$ and we are able to compare the distributions of
representations of different groups.

\subsubsection{The invariant spaces for the unitary groups}
\label{subsec:intro-invariant-spaces}
In the concrete example
of the series of the unitary groups $G_n=U(n)$, the corresponding
invariants are given by the vector spaces given by the symmetric groups algebras
$\C[\Sy{\rank}]$, as shown in Section \ref{subsec:invariant-spaces}.
Consider a representation $\rho_n:\uu(n)\rightarrow \End(V_n)$ 
of the Lie algebra $\uu(n)$ of the unitary group $U(n)$.
The corresponding moment 
\[ \momentTensor_\rank(\varepsilon_n \rho_n) \in \left[ \big( \uu(n)\gwia
\big)^{\otimes \rank}
\right]_{U(n)} \subseteq \C[\Sy{\rank}]\]
can be identified with a function on the symmetric group $\Sy{r}$ which is given
explicitly (for $n\geq r$) as 
\begin{equation} 
\label{eq:moments_moments}
\big( \momentTensor_\rank(\varepsilon_n \rho_n) \big) (\pi) = \varepsilon_n^\rank
\trVn\left[ \rho_n (
e_{1 \pi_1} ) \cdots \rho_n ( e_{\rank \pi_\rank} ) \right],
\end{equation}
where $e_{ij}\in \M_n(\C)= \uu(n) \otimes_\R \C$ are the matrix units
(we will show a refined version of this in Proposition \ref{prop:7}).
The above quantities \eqref{eq:moments_moments} contain complete information
about representation $\rho_n$; the study of asymptotics of representations is
therefore reduced to studying asymptotics of $\momentTensor_\rank(\varepsilon_n
\rho_n)\in\C[\Sy{\rank}]$ in the
limit $n\to\infty$. It remains to determine which asymptotics will be most
convenient.

\subsubsection{Higher order free probability}
The same problem appears in the random matrix theory, where analogous quantities
$ \momentTensor_\rank(\genericMatrix_n)$ can be considered for a unitarily invariant random
matrix $\genericMatrix_n$. This problem
has been studied in the context of the theory of \emph{higher order free probability}
which was introduced by Mingo and Speicher and later on was further developed
also by the authors of this article
\cite{HigherOrderFreeness1,HigherOrderFreeness2,HigherOrderFreeness3}. The main
goal of this theory is to give an abstract framework which would be able to
describe asymptotics of fluctuations of random matrices in a similar way as
Voiculescu's original free probability \cite{VoiculescuDykemaNica} describes the
mean behavior of random matrices. This goal was achieved by the notions of
\emph{higher order moments} and \emph{higher order free cumulants} which on one side have very
nice probabilistic interpretations for a given sequence of random matrices and
on the other side are abstract quantities which concern abstract objects
modeling limits of random matrices. 

The current paper gives applications of the combinatorial machinery of higher order free probability \cite{HigherOrderFreeness3} to
representation theory, and therefore stands as a first application of higher order freeness beyond random matrix theory.

\subsection{Organization of the paper}

In Section \ref{sec:noncommutative}, we recall the notations related to
non-commutative random variables and non-commutative random vectors.
In Section \ref{sec:unitarily}, we study unitarily invariant random matrices
with non-commutative entries.
In Section \ref{sec:non-commuting-entries}, we study representations as random
matrices with non-commutative entries and prove our main result (Theorem \ref{theo:main_theorem} 
will be proved in a more precise formulation as Theorem \ref{theo:main_theorem-RELOADED}).
In Section \ref{sec:applications}, we present applications of the main result
and proofs of the results presented in Section \ref{subsec:concrete-applications}.

\section{Non-commutative probability}
\label{sec:noncommutative}

\subsection{Traces}
We denote by $\Tr$ the usual trace on the matrix algebra $\M_n(\C)$ and by
$\trn:=\frac{1}{n} \Tr$ the \emph{normalized trace}. For an endomorphism $x\in\End(V)$, 
we denote by $\trV x = \frac{1}{\text{dim $V$}} \Tr x$ the corresponding normalized trace.

With these notations, the traces of the unit 
matrix $1\in\M_n(\C)$ are given by
\begin{align*}
  \Tr 1 & = n, \\
  \trn 1 & = 1.
\end{align*}

\subsection{Non-commutative probability spaces}
\label{subsec:noncommutative-probability}
Let us recall briefly some basic notions of \emph{non-commutative probability theory} 
\cite{VoiculescuDykemaNica,MeyerQuantumprobability}.

Let $( \Omega, \mathfrak{M}, P )$ be a Kol\-mo\-go\-rov probability space. We
consider the algebra
\[ \El^{\infty -} ( \Omega ) := \bigcap_{n\geq 1} \El^n(\Omega)\]
of random variables with all moments finite. This algebra is equipped with a
functional $\E : \El^{\infty -} ( \Omega ) \rightarrow \C$ which to a random
variable associates its mean value.

We consider a generalization of the above setup in which the commutative algebra
$\El^{\infty -} ( \Omega )$ is replaced by any (possibly
non-com\-mu\-ta\-ti\-ve) \mbox{$\star$-algebra} $\A$ with a unit and $\E : \A
\rightarrow \C$ is any linear functional which is normalized (i.e., $\E ( 1 ) =
1$) and positive (i.e., $\E(x\gwia x)>0$ for all $x\in\A$ such that $x\neq 0$).
The elements of $\A$ are
called \emph{non-commutative random variables} and the functional $\E$ is
called the \emph{mean value} or \emph{expectation}. We also say that $( \A, \E )$ is a 
\emph{non-commutative probability space}.

The \emph{joint distribution} of a family $(x_i)_{i\in I}$ of non-commutative random variables is defined as 
the collection of their moments $\big(\E (x_{i_1} \cdots x_{i_l}) \big)_{i_1,\dots,i_l\in I}$. Classical random variables 
can also be viewed as non-commutative random variables; notice that the concept of the (joint) 
distribution of random variables is different in both setups but in the case of probability measures which are uniquely 
determined by their moments 
both notions determine each other.

\subsection{Partitions and partitioned permutations}
We recall briefly basic combinatorial tools of higher order free probability theory, in particular 
\emph{partitioned permutations} \cite[Section 4]{HigherOrderFreeness3}.

The set of \emph{partitions} of the set $[\rank]:=\{1,\dots,\rank\}$ is endowed with the partial order
defined as follows:  $\V\leq\W$ if every block of partition $\V$ is
contained in some block of partition $\W$. 

For a permutation $\pi\in\Sy{r}$ we denote by $C(\pi)$ the partition of $[r]$ corresponding to
the cycles of $\pi$. We write $\pi\leq\W$ if every cycle of the permutation
$\pi$ is contained in some block of the partition $\W$ or, in other words, if
$C(\pi)\leq \W$.

We denote by $\#\V$ the number of blocks of a partition $\V$. We also denote by
$\#\pi=\#C(\pi)$ the number of cycles of $\pi$.

The set of partitions carries a lattice structure $\vee,\wedge$, where the
smallest element is the discrete partition
$\mathbf{0}=\mathbf{0}_\rank:=\big\{\{1\},\ldots ,\{\rank\}\big\}$ and the largest element
is the rough partition
$\mathbf{1}=\mathbf{1}_\rank:=\big\{\{1,\ldots ,\rank\}\big\}$.

A \emph{partitioned permutation} of $[\rank]$ is a pair $(\V,\pi)$,
where $\V$ is a partition of $[\rank]$ and $\pi$ is a permutation of 
the same set $[\rank]$ such that $\pi\leq \V$. For a given permutation $\pi$ we denote  by
$(\mathbf{0},\pi):=(C(\pi),\pi)$ the partitioned permutation with the smallest
possible partition for $\pi$.

We define the \emph{length of the permutation} $\pi\in \Sy{\rank}$ as $|\pi|:=\rank-\#\pi$.
We also define the \emph{length of the partitioned permutation}
$(\V,\pi)$ of the set $[\rank]$ as 
\[ |(\V,\pi)|:= |\pi|+2 (\#\pi - \#\V) \]
and the \emph{length of a partition} $|\V|$ of the same set as $|\V|:=\rank-\#\V$.

We say that $(\V_1,\pi_1) \cdot (\V_2,\pi_2)= (\V_3,\pi_3)$ if
$\V_1\vee\V_2=\V_3$ and $\pi_1 \pi_2=\pi_3$ and $|(\V_1,\pi_1)| +
|(\V_2,\pi_2)|=|(\V_3,\pi_3)|$. Notice that with this definition the product of
two partitioned permutation is not always defined. 

We say that $(\V_1,\pi_1) \leq (\V_2,\pi_2)$ if $(\V_1,\pi_1) \cdot
(\mathbf{0},\pi_1^{-1} \pi_2) = (\V_2,\pi_2)$. This relation is, in
general, \emph{not transitive}. However, since $(\V_1,\pi_1) \leq (\V_2,\pi_2)$ implies $|(\V_1,\pi_1)| \leq |(\V_2,\pi_2)|$ and 
the latter inequality becomes an equality only if $(\V_1,\pi_1)=(\V_2,\pi_2)$, it follows that the relation $\leq$ in the set of 
partitioned permutations is \emph{acyclic} and thus can be extended to a linear order.

The symmetric group $\Sy{\rank}$ acts on the set of partitions of $[\rank]$ as follows: for $\pi\in\Sy{\rank}$ and partition 
$\V$ of $[\rank]$ we define $\pi(\V)$ as the unique partition which connects the elements $\pi(a)$ and $\pi(b)$ 
if and only if $a$ and $b$ are connected by partition $\V$, for arbitrary $a,b\in[\rank]$.

We say that partitioned permutations $(\V_1,\pi_1)$ and $(\V_2,\pi_2)$ are
\emph{conjugate} by a permutation $\sigma$ if they are equal after relabeling the
elements of $[\rank]$ given by $\sigma$. Formally speaking, this means
that $\pi_2= \sigma \pi_1 \sigma^{-1}$ and $\sigma(\V_1)=\V_2$.

\subsection{Tensor independence and non-commutative cumulants}
\label{subsec:tincc}

Let $(\A_i)$ be a (finite or infinite) sequence of subalgebras of some
non-commutative probability space $\A$. They are said to be \emph{tensor independent}
if they commute and $\E(a_1 a_2 \cdots )=\E(a_1) \E(a_2) \cdots$
holds for all sequences $(a_i)$ which contain only finitely many elements
different from $1$ and such that $a_i\in\A_i$. Tensor independence can be
regarded as a substitute of the usual independence of classical random variables
in the non-commutative setup.

Let $\widetilde{\A}=\bigotimes_{n\in\mathbb{N}} \A$ be the inductive limit of
algebraic tensor products. This is a non-commutative probability space together
with the infinite tensor product state $\E^{\otimes \infty}$. Clearly, the 
subalgebras 
\[\A^{(i)}:=1^{\otimes i-1}\otimes \A\otimes 1 \otimes \cdots \subset \widetilde{\A} \]
are tensor independent. We will regard $(\A^{(i)})_i$ as a family of
tensor independent copies of the algebra $\A$. 
Given $a\in\A$, we define its $i$-th tensor independent copy
$a^{(i)}\in  \A^{(i)}$ by 
\[a^{(i)}:=1^{\otimes i-1}\otimes a\otimes 1\otimes\cdots. \]

With this material we can introduce the notion of a \emph{non-commutative
cumulant}.
For each $i\in[\rank]$ let $a_i\in\A$ be a non-commutative
random variable. For any partition $\V$ of $[\rank]$ we can define a multilinear moment map
\[\E_{\V}: \underbrace{\A\times \cdots \times\A}_{\text{$\rank$ times}}\to \C\] by
\[\E_{\V}(a_1,\dots,a_\rank)=
\E^{\otimes \infty} \left(a_1^{(b(1))} \cdots  a_\rank^{(b(\rank))}\right), \]
where $b: [\rank]\to \mathbb{N}$ is any function defining the partition
$\V$, i.e., $b(i)=b(j)$ if and only if $i$ and $j$ belong to the same block of
$\V$. Following the classical scheme \cite{Lehner2002cumulantsI}, 
we define \emph{tensor cumulants} to be the unique multilinear maps 
\[k_{\V}: \underbrace{\A\times \cdots \times\A}_{\text{$\rank$ times}} \to \C\]
(where $\V$ is a partition of $[\rank]$) such that
\begin{equation}
\label{eq:moment-cumulant-formula}
\sum_{\W\leq \V}k_{\W}=\E_{\V} 
\end{equation}
for every partition $\V$.

A special role is played by the cumulant corresponding to the maximal
partition; we will use a special notation for it:
\[ k_\rank(a_1,\dots,a_\rank):=k_{\mathbf{1}_\rank}(a_1,\dots,a_\rank).  \]

Observe that this definition is actually Lehner's cumulant in case of the tensor
independence case, cf.~\cite{Lehner2002cumulantsI}. 
When $\A=\mathcal{L}^{\infty-}(\Omega)$, this corresponds to the classical
probability space, and tensor cumulants coincide with the classical cumulants of
random variables.

Notice that the family $(\E_{\V})$ is \emph{multiplicative} in the sense that
$\E_{\V}(a_1,\dots,a_\rank)$ is a product of the expressions 
$\E(a_{i_1}\cdots a_{i_m})$ over the blocks
$\{i_1<\cdots<i_m\}$ of
the partition $\V$. It follows immediately that the family $(k_{\V})$ is
multiplicative as well.
For more on this
topic of multiplicative functions on
partitions and their applications to free probability theory we refer to
\cite{NicaSpeicher_book}.

\subsection{Cumulants and commutators}

In the following we  use the following notational shorthands:
\begin{align*} k_n(\dots,a_i,a_{i+1},\dots)
& :=k_n(a_1,\dots,a_{i-1},a_i,a_{i+1},a_{i+2},\dots,a_n),\\
k_n(\dots,a_{i+1},a_{i},\dots) & :=k_n(a_1,\dots,a_{i-1},a_{i+1},a_i,a_{i+2},\dots,
a_n),\\
k_{n-1}(\dots,[a_{i},a_{i+1}],\dots) &
:=k_{n-1}(a_1,\dots,a_{i-1},[a_{i},a_{i+1}],a_{i+2},\dots,
a_n),
\end{align*}
where $[x,y]=xy-yx$ denotes the commutator, and similar ones.

\begin{lemma}
\label{lem:cumulants-and-commutators}
For any elements $a_1,\dots,a_\rank\in\A$, any $1\leq i\leq \rank-1$ and any partition
$\W$ of $[\rank]$,
\begin{multline}
\label{eq:commutator}
k_{\W}(\dots,a_i,a_{i+1},\dots)-k_{\pi(\W)}(\dots,a_{i+1},a_{i},
\dots) = \\
\begin{cases}
0 & \text{if $i$ and $i+1$ are not connected by $\W$},\\
k_{\W'} (\dots,[a_i,a_{i+1}],\dots) & \text{otherwise},
\end{cases}
\end{multline}
where $\pi=(i,i+1)\in\Sy{\rank}$ denotes the transposition interchanging $i$ and $i+1$,
and where $\W'$ denotes the partition of\/ $[\rank-1]$ resulting from $\W$ by
merging $i$ and $i+1$ into one element $i$ and by relabeling the elements
$i+2,\dots,\rank$ into the elements $i+1,\dots,\rank-1$.
\end{lemma}
\begin{proof}
We split the proof into two parts.
\begin{enumerate}[label=\emph{\alph*})]
 \item 
\label{part:not-connected}
Let us consider the case when  $i$ and $i+1$ are not connected by $\W$. 
We use M\"obius inversion in a rather weak form, i.e.,~the fact that  the cumulant $k_\W$ is a linear combination of moments $\E_\W$ over $\W\leq \V$; thus over $\W$ which do not connect $i$ with $i+1$. We compare such expressions for each of the two 
terms on the left-hand side of \eqref{eq:commutator}; they clearly coincide.

\item
Let us consider now the case when  $i$ and $i+1$ are connected by $\W$. 
Roughly speaking, the proof is an application of the non-commutative version of the formula of Leonov and Sirjaev \cite{LeonovSiraev} for cumulants of products to the right
hand side of the above equality. We provide the details of the proof below.

Let $\V$ be any partition of $[\rank]$ such that $i$ and $i+1$ are
connected by $\V$. From the defining relations for cumulants it follows that
\begin{multline*}\E_{\V'}(\dots,[a_i,a_{i+1}],\dots)=\\ \E_\V(\dots,a_i,a_{i+1},
\dots)-\E_\V(\dots , a_{i+1},a_{i},\dots)= \\ 
\sum_{\W\leq\V} k_\W(\dots,a_i,a_{i+1},\dots)-
\sum_{\W\leq\V} k_\W(\dots,a_{i+1},a_{i},\dots)= \\
\sum_{\W\leq\V} k_\W(\dots,a_i,a_{i+1},\dots)-
k_{\pi(\W)}(\dots,a_{i+1},a_{i},\dots),
\end{multline*}
where in the last equation we used the fact that $\W\mapsto \pi(\W)$ is a permutation of the set of partitions which are smaller than $\V$.

From the case \ref{part:not-connected} considered above it follows that if $\W$ does not connect $i$ and $i+1$ then the corresponding summand on the right hand side is equal to
zero. It follows that the sum on the right hand side can be written as
\[ \sum_{\W'\leq \V'}
k_\W(\dots,a_i,a_{i+1},\dots)-k_\W(\dots,a_{i+1},a_{i},\dots), \]
where $\W$ is the partition of $[\rank]$ with a property that $i$ and $i+1$ are connected
by $\W$, obtained from $\W'$ by splitting the element $i$ into $i$ and $i+1$ and by relabeling the elements
$i+1,\dots,\rank-1$ into the elements $i+2,\dots,\rank$.

It follows that the function on the set of partitions of $[\rank-1]$ defined by 
\[ k_{\W'}:=k_\W(\dots,a_i,a_{i+1},\dots)-k_\W(\dots,a_{i+1},a_
{i},\dots) \]
fulfills the defining property \eqref{eq:moment-cumulant-formula} of cumulants
$k_{\W'} (\dots,[a_i,a_{i+1}],\dots)$. Since such a function is unique, this finishes the proof.
\end{enumerate}
\end{proof}

\subsection{Non-commutative random vectors}
\label{subsect:random_vectors}

Let $( \A, \E )$ be a non-com\-mu\-ta\-ti\-ve probability space and $W$ be a
vector space; the elements of $W \otimes \A$ will be called
\emph{non-com\-mu\-ta\-ti\-ve random vectors in $W$ (over the non-commutative
probability space $( \A, \E )$)}.

Given elementary tensors $w_1=x_1\otimes a_1\in W_1\otimes \A$ and $w_2=x_2\otimes a_2 \in
W_2\otimes \A$, we define 
\[w_1\widehat{\otimes} w_2=(x_1\otimes a_1)\widehat{\otimes} (x_2\otimes a_2):=
(x_1\otimes x_2\otimes a_1a_2)\in W_1\otimes W_2\otimes \A
\]
and its linear extension to non-elementary tensors.
Whenever $w_1=w_2$ with $W_1=W_2$ one shortens the notation 
as $w^{\widehat{\otimes} 2}\in W^{\otimes 2}\otimes \A$ and one extends it 
by induction 
to the definition of
\[w^{\widehat{\otimes} \rank}\in W^{\otimes \rank}\otimes \A .\]
Observe that this definition is reminiscent of the definition of tensor product
of representations 
of compact quantum groups of Woronowicz
\cite{Woronowicz87Compact-matrix-pseudogroups}
provided that $\A$ is a quantum group and $W$ a representation
of $\A$.

For a non-commutative random vector $w$ we define its \emph{$\rank$-th order
vector moment} $\momentTensor_\rank(w)$ to be
\[\momentTensor_\rank(w):=(\Id\otimes  \E )w^{\widehat{\otimes} \rank}\in 
W^{\otimes \rank}.\]
We define the \emph{distribution of a non-commutative random vector} as the sequence 
$(\momentTensor_\rank(w))_{\rank=1,2,\dots}$ of its moments. 
These moments can be used in the obvious way to define 
 \emph{convergence in moments of non-commutative random vectors}.

The above definitions can be made more explicit as follows: let
$e_1,\dots,e_d$
be a base of the finite-dimensional vector space $W$. Then a (classical)
random vector $w$ in $W$ can be viewed as
\begin{equation}
\label{eq:wspolrzedne}
 w= \sum_i a_i e_i,
\end{equation}
where $a_i$ are the (random) coordinates. Then a non-commutative random vector
can be viewed as the sum \eqref{eq:wspolrzedne}, in which $a_i$ are replaced by
non-commutative random variables. One can easily see that the moment
\[ \momentTensor_\rank(w) = \sum_{i_1,\dots,i_\rank} \E(a_{i_1} \cdots a_{i_\rank}) \   
e_{i_1} \otimes \cdots \otimes e_{i_\rank}\]
contains nothing else but the information about the mixed moments of the
non-commutative coordinates $a_1,\dots,a_d$ and the convergence of moments of $w$ is
equivalent to the convergence of the mixed moments of $a_1,\dots,a_d$.

In the sequel of the paper, we pay special attention to the case when the
vector space $W=\M_n(\C)$ is the matrix algebra. In this case the
non-commutative random vectors, elements of $\M_n(\C) \otimes \A=\M_n(\A)$ can
be also called \emph{random matrices with non-commutative entries}.

\subsection{Non-commutative probability space corresponding to the Lie algebra representation}
\label{subsec:lie-representation-random-matrix-A}

In this article we  concentrate on the following example of a non-commutative random vector related to a representation of some Lie algebra.

Let $\g$ be a finite-dimensional Lie algebra. Its representation
$\rho:\g\rightarrow\End(V)$
can be alternatively viewed as $\rho\in \g\gwia \otimes \End(V)$, i.e., as a
non-commutative random vector in $\g\gwia$ (the vector space dual to the vector space $\g$) over the non-commutative
probability space $\big(\End(V),\trV \big)$. In this case 
the moment $ \momentTensor_\rank(\rho) \in (\g\gwia)^{\otimes \rank}$ can be alternatively viewed as
$ \momentTensor_\rank(\rho): \g^{\otimes \rank} \rightarrow\C$ which is given explicitly on elementary tensors by
\[ \momentTensor_\rank(\rho) (x_1 \otimes \cdots \otimes x_\rank) = \trV \big[  \rho(x_1) \cdots \rho(x_\rank) \big]
\qquad \text{for } x_1,\dots,x_\rank\in\g. \]

We consider the coadjoint action of $G$ on $\g\gwia$ by the complex conjugate matrix\footnote{
In the case when $G=U(n)$, the definition of the complex conjugate $\bar{g}$ creates no difficulties. However, for an abstract group $G$, this complex conjugate might be not well defined. In this case one should rather consider the usual coadjoint action
$g\cdot x:=(\Ad_{{g}^{-1}})\gwia(x)$. Note that the sets of $G$-invariant tensors for both actions are identical.}, i.e.~the action 
given explicitly by 
\[g\cdot x:=(\Ad_{\bar{g}^{-1}})\gwia(x)=(\Ad_{g^T})\gwia(x)\] 
for $g\in G$ and
$x\in\g\gwia$. This action extends to an action of $G$ on $(\g\gwia)^{\otimes
\rank}$.

\begin{lemma}
\label{lem:invariance-of-moments}
If $\rho:\g\rightarrow\End(V)$ is a representation viewed as a
non-com\-mu\-ta\-tive
random vector and $\rank\geq 1$ is an integer then 
\[ \momentTensor_\rank(\rho) \in \left[ (\g\gwia)^{\otimes \rank} \right]_G,\]
i.e., it is invariant under the coadjoint action of $G$. 
\end{lemma}
\begin{proof}
For any $x_1,\dots,x_\rank\in\g$ and $g\in G$
\begin{align*}  \big( g \cdot \momentTensor_\rank(\rho) \big) (x_1\otimes \cdots
\otimes x_\rank) 
= & \momentTensor_\rank(\rho)  \big( \Ad_{\bar{g}^{-1}}(x_1)\otimes \cdots
\otimes \Ad_{\bar{g}^{-1}}(x_\rank) \big) \\
= &
\trV \left[ \rho\big(\Ad_{\bar{g}^{-1}}(x_1) \big)  \cdots  \rho\big(\Ad_{\bar{g}^{-1}}(x_\rank)
\big) \right] \\
= &
\trV \left[ \rho(\bar{g}^{-1}) \rho(x_1)  \cdots \rho(x_\rank)
 \rho(\bar{g}) \right] \\ 
= & \trV \left[ \rho(x_1)  \cdots \rho(x_\rank)
\big)  \right] \\
= & \momentTensor_\rank(\rho) (x_1\otimes \cdots
\otimes x_\rank). 
\end{align*}
By linearity, the above equation extends to general tensors.
Thus we have shown that $ g \cdot \momentTensor_\rank(\rho)= \momentTensor_\rank(\rho)$ as required.
\end{proof}

\section{Unitarily invariant matrices with non-commutative entries and higher-order
probability spaces}
\label{sec:unitarily}

For the case when $G=U(n)$ is the unitary group and $\g=\uu(n)$ is its Lie algebra, 
we elaborate on the discussion of Section
\ref{subsec:elements-of-proof-II} and describe the invariant space
$[(\g\gwia)^{\otimes \rank}]_{G}$ to which the moments $\momentTensor_\rank(\rho)$
belong.

\subsection{The matrix structure on $\uu(n)\gwia$}
\label{subsec:matrix-structure-1}

We equip the linear space $\uu(n)$ of antihermitian matrices with a non-degenerate bilinear symmetric form
\begin{equation}
\label{eq:bilinear-form}
\langle x,y \rangle=\Tr x^T y
\end{equation}
which gives an isomorphisms allowing to identify $\uu(n)\gwia$ with $\uu(n)$.
This isomorphism is equivariant with respect to the action of the unitary group $U(n)$; indeed, if $f\in \uu(n)\gwia$ is the functional corresponding to $x\in \uu(n)$, then for any $y\in \uu(n)$ and $g\in U(n)$
\[
(g\cdot f)(y)= f\left( \Ad_{\bar{g}^{-1}}(y) \right)
=\Tr x^T  \bar{g}^{-1} y \bar{g} =
\Tr \left(  {g} x {g}^{-1} \right)^T y =
\left\langle \Ad_{g}(x), y \right\rangle 
\]
which shows that indeed $g\cdot f$ is the functional corresponding to $g\cdot x$.

Note that the Lie
algebra complexification $\uu(n) \otimes_\R \C=
\gl(n) = \M_n(\C)$ has a matrix structure and thus
$\uu(n)\gwia \otimes_\R \C\cong \uu(n) \otimes_\R \C = \M_n(\C)$
can be identified with matrices. This identification has the following concrete form:
$x\in \uu(n)\gwia \otimes_\R \C$ corresponds to the matrix
\[ \begin{bmatrix}
x(e_{11}) & \dots  & x(e_{n1}) \\
\vdots       & \ddots & \vdots       \\
x(e_{1n}) & \dots  & x(e_{nn}) 
\end{bmatrix} = \sum_{i,j} x(e_{ij})\ e_{ij} \in \M_n(\C),\]
where $e_{ij}\in \M_n(\C) = \uu(n) \otimes_\R \C$ are the matrix units.
Indeed, the above matrix defines via \eqref{eq:bilinear-form} a functional which on a matrix unit $e_{kl}$ takes the same value as the functional $x$:
\[ \left\langle  \sum_{i,j} x(e_{ij})\ e_{ij}, e_{kl} \right\rangle =
x(e_{kl}).\]

\subsection{The invariant spaces}
\label{subsec:invariant-spaces}

We will need following classical result, known as Schur-Weyl duality theorem \cite[Section 9.1]{GoodmanWallach}.

\begin{theorem}
\label{theo:generalized-schur-weyl}
Let $\rho$ be the diagonal action of the unitary group $U(n)$ on
$(\C^n)^{\otimes \rank}$. Let $\tilde{\rho}$ be the action of the symmetric group
$\Sy{\rank}$ on $(\C^n)^{\otimes \rank}$ by permutation of elementary
tensors. 

The actions of $\Sy{\rank}$ and of\/ $U(n)$ commute, therefore
$\rho\times\tilde{\rho}$ is a representation of\/ $\Sy{\rank}\times U(n)$
on $(\C^n)^{\otimes \rank}$. This representation is multiplicity free. 
Equivalently, the commutant of $\rho$ in $(\C^n)^{\otimes \rank}$ is $\tilde{\rho}$ and vice versa.
\end{theorem}

From the identification from Section \ref{subsec:matrix-structure-1}
it follows that we can view any 
$Z\in\left[\big(\uu(n)\gwia)^{\otimes \rank} \right]_{U(n)}$ as an 
endomorphism of $(\C^n)^{\otimes \rank}$ which commutes with the diagonal action of
$U(n)$. From Schur-Weyl duality (Theorem \ref{theo:generalized-schur-weyl}) it
follows that $Z$ can be identified with an element of the symmetric group algebra $\C[\Sy{\rank}]$.

Thus we have shown that
\[ \left[ \big( \uu(n)\gwia \big)^{\otimes \rank} \right]_{U(n)} \subseteq \C[\Sy{\rank}], \]
just as we claimed in Section \ref{subsec:intro-invariant-spaces}.
In fact, if we replace the left-hand side by its complexification and assume that $\rank\leq n$ then the equality holds, but we will not need this more general result.

\subsection{Unitarily invariant classical random matrices and their random moments}

For $Z\in\left[\End\big( (\C^n)^{\otimes \rank} \big)\right]_{U(n)}$,
we consider the function $\Tr_{\bullet} Z
\in \C[\Sy{\rank}]$ defined by
\[ \Tr_{\sigma} Z := \Tr (\sigma Z) \qquad \text{for any }\sigma\in\Sy{\rank},\]
where on the right-hand side we view $\sigma$ as an
endomorphism of $(\C^n)^{\otimes \rank}$ given by permutation of the factors. 
It is known --- see, for example \cite{CollinsSniady2004} --- that
$\Tr_{\bullet} Z$ gives a complete information about $Z$.

If a $U(n)$-invariant (classical) random element $\naturalMatrix$ in $\uu(n)\gwia$
is viewed as a random matrix in $\uu(n)\gwia \otimes_\R \C=\M_n(\C)$, then
$\Tr_{\bullet} \naturalMatrix^{\otimes \rank}$ is a
function on the symmetric group (with values being random variables).
It is
central and multiplicative with respect to the cycle decomposition of
permutations; it follows that the family
$\big(\E \Tr_{\sigma} \naturalMatrix^{\otimes \rank}\big)_{\sigma\in\Sy{r},\ r=1,2,\dots}$ can be interpreted as the
collection of mixed moments of the random variables corresponding to the
cycles $(1,\dots,\rankB)\in\Sy{\rankB}$, $s=1,2,\dots$:
\begin{equation}
\label{eq:strange-quantities}
\Tr_{(1,\dots,\rankB)} \left(\naturalMatrix^{\otimes s}\right)= \Tr \naturalMatrix^\rankB = n\ \trn \naturalMatrix^\rankB=
n \int_\C z^\rankB \ d\spectralNatural_\naturalMatrix.
\end{equation}
Notice that the
definition of the spectral measure
$\spectralNatural_\naturalMatrix$ has to be modified for $\naturalMatrix\in\uu(n)\gwia$, since
the latter corresponds to an antihermitian matrix, and therefore its spectral measure is
supported not on the real line $\R$, but on the imaginary line $i\R$.
In other words, \emph{all the information about the distribution of $\naturalMatrix$} 
(from the viewpoint of non-commutative probability theory) 
\emph{is contained in the
family of random variables \eqref{eq:strange-quantities}}.
The above quantities \eqref{eq:strange-quantities} are random variables which
have a very simple interpretation as random moments of the spectral measure of
$\naturalMatrix$ viewed as a random matrix. 
\emph{Thus the study of a unitarily invariant (classical)
random element in $\uu(n)\gwia$ reduces to studying the joint
distribution of the family \eqref{eq:strange-quantities} or, equivalently, to
studying the behavior of its random spectral measure $\spectralNatural_\naturalMatrix$.}

In this article we are concerned about a \emph{non-commutative} random vector in
$\uu(n)\gwia$ which corresponds to some representation of Lie algebra
$\uu(n)$; due to this noncommutativity, the discussion from the previous
paragraph does not apply directly. However, the scaling of the representations
considered in this article is such that asymptotically this noncommutativity
becomes in some sense negligible, therefore the spectral measure
$\spectralNatural_\naturalMatrix$ and its moments still remain very useful notions. 
Nevertheless we need to explain how to define the spectral measure for a
random matrix with non-commutative entries and we shall do it in the following.

\subsection{Random matrices with non-commutative entries and their spectral
measures}
\label{subsec:random-matrices-noncommutative}
Let $\naturalMatrix\in\M_n(\A)$ be a hermitian 
random matrix with non-commutative entries. 
If the joint distribution of the 
non-commutative random variables 
\[\left(\trn \naturalMatrix^\rank\right)_{\rank=1,2,\dots}\] 
coincides in the sense of non-commutative probability theory (i.e., the mixed moments coincide) with
the joint distribution of classical random variables of the form 
\[\left( \int_\R z^\rank \ d\spectralNatural_\naturalMatrix\right)_{\rank=1,2,\dots},\]
where $\spectralNatural_\naturalMatrix$ is a random
probability measure on $\R$, we
say that $\spectralNatural_\naturalMatrix$ is the \emph{(natural) spectral measure} of $\naturalMatrix$.

Clearly, for classical random matrices the above definition coincides with the
usual definition of the spectral measure \eqref{eq:definition-spectral-measure} under assumption that the joint distribution of traces (this time viewed as a probability measure) is uniquely determined by its moments.
In the general non-commutative case the existence and the uniqueness of the spectral measure are not obvious.

\subsection{Unitarily invariant random matrices}
Let $(\A,\E)$ be a non-com\-mu\-ta\-tive probability space. We say that
a random matrix with non-commutative entries $\naturalMatrix\in\M_n(\A)$ is \emph{unitarily
invariant} if for every $U\in U(n)$ the joint distribution of the entries the
matrix $\naturalMatrix=(y_{ij})_{1\leq i,j\leq n}$ coincides with the joint distribution of
the entries of the matrix $\naturalMatrix'=(y'_{ij})_{1\leq i,j\leq n}:=U \naturalMatrix U^{-1}$.

In the following, we use the notation
\[[\text{\emph{condition}}]=
\begin{cases}
 1 & \text{if \emph{condition} is true},\\
 0 & \text{otherwise.}
\end{cases}
\]

\begin{proposition}
\label{prop:7}
If\/ $\naturalMatrix\in\M_n(\C)\otimes \A$ is a unitarily invariant $n\times n$ random matrix with
non-commutative entries then for each integer $1\leq \rank\leq n$ and each partition
$\V$ of the set $[\rank]$, there exists a unique function $\Sy{r}\ni\pi \mapsto \kappa_{(\V,\pi)}\in\C$ with the property that for all choices of the indices
$i_1,\dots,i_\rank,j_1,\dots,j_\rank\in [n]$, we have
\begin{equation} 
\label{eq:definition-of-kappa}
k_\V \big( \naturalMatrix_{i_1 j_1}, \dots, \naturalMatrix_{i_\rank j_\rank} \big)= 
   \sum_{\pi\in \Sy{\rank}}  [j_1=i_{\pi(1)}] \cdots [j_\rank=i_{\pi(\rank)}]\
\kappa_{(\V,\pi)}.
\end{equation}
Furthermore, $\kappa_{(\V,\pi)}$ is non-zero only for $\pi\leq \V$. 


This function is explicitly given by 
\begin{equation} 
\label{eq:kappa_concrete}
\kappa_{(\V,\pi)} = k_\V \big( \naturalMatrix_{1  \pi(1)}, \dots, \naturalMatrix_{\rank
\pi(\rank) } \big).
\end{equation}
\end{proposition}
\begin{proof}
By rearranging the factors we may view $\naturalMatrix^{\otimes \rank}\in \big(\M_n(\C)\big)^{\otimes
\rank} \otimes \A^{\otimes \rank}$. 
The multilinear maps $\E_{\V}$ and $k_{\V}$ give rise to linear 
functionals $\E_{\V}:\A^{\otimes \rank}\rightarrow \C$ and $k_{\V}:\A^{\otimes \rank}\rightarrow \C$. 
The assumption that $\naturalMatrix$ is unitarily invariant implies that 
the element $(\Id\otimes \E_{\V})(\naturalMatrix^{\otimes \rank})  \in \big( \M_n(\C) \big)^{\otimes \rank}$
is invariant under the adjoint action of the unitary group for arbitrary partition $\V$ of $[\rank]$; 
it follows that $(\Id\otimes k_{\V})(\naturalMatrix^{\otimes \rank})  \in \big( \M_n(\C) \big)^{\otimes \rank}$ is invariant as well. 

From Schur-Weyl duality (Theorem
\ref{theo:generalized-schur-weyl}) it follows that $(\Id\otimes k_{\V})(\naturalMatrix^{\otimes \rank})$  can be identified with an element of the symmetric group algebra which will be denoted by $\kappa_\V\in\C[\Sy{\rank}]$. If we view this element as a function $\kappa_\V: \Sy{r} \ni \pi \mapsto \kappa_{(\V,\pi)}\in\C$ and calculate
\[ \Tr \Big[ \left( e_{j_1 i_1} \otimes \cdots \otimes e_{j_\rank i_\rank} \right)
\left[  (\Id\otimes k_{\V})(\naturalMatrix^{\otimes \rank})  \right] \Big] \]
in two different ways then the equality
\eqref{eq:definition-of-kappa} follows immediately. Thus we proved existence of the function $\kappa_{(\V,\pi)}$.

Equation
\eqref{eq:kappa_concrete} follows by appropriate choice of the indices in
\eqref{eq:definition-of-kappa}; thus we also proved uniqueness of $\kappa_{(\V,\pi)}$.

Assume that $\pi\not\leq\V$, then by multiplicativity the right-hand side of \eqref{eq:kappa_concrete} can be written as a product of expressions of the form 
$k_\rankB( \naturalMatrix_{i_1 j_1}, \dots, \naturalMatrix_{i_\rankB j_\rankB})$ and for each such an expression an analogue of \eqref{eq:definition-of-kappa} holds true  
as well. For the right-hand side of \eqref{eq:definition-of-kappa} to be non-zero we must have the equality of the multisets $(i_1,\dots,i_s)$ and $(j_1,\dots,j_s)$ which would imply that $\pi\leq \V$ which contradicts $\pi\not\leq\V$. Thus $\kappa_{(\V,\pi)}=0$ as claimed.
\end{proof}

\subsection{Higher order free probability}
The concept of \emph{higher order free probability} was introduced in a series of
papers \cite{HigherOrderFreeness1,HigherOrderFreeness2,HigherOrderFreeness3}.
In this article we deal with a simplified problem of fluctuations of a single
random matrix (as opposed to fluctuations of several random matrices). In this
section, we present the necessary notions and notations of higher order free
probability in this simplified setup.

Assume that for each $n\geq1$, an $n\times n$ random matrix $\naturalMatrix^{(n)}$ with non-commutative entries is given.
When there is no possible confusion, we omit the explicit
dependence on $n$ and we will simply write $\naturalMatrix=\naturalMatrix^{(n)}=(y_{ij})_{1\leq i,j\leq
n}$. We systematically assume that $\naturalMatrix$ is unitarily invariant.

Two kinds of quantities can be used to describe properties of the random matrix
$\naturalMatrix$. The \emph{macroscopic} quantities describe the probabilistic behavior of
the family of the traces $(\Tr \naturalMatrix^\rank)_{\rank\geq 1}$. We are  interested, up to some
normalization, in the tensor cumulants of the form:
\begin{equation}
\label{eq:cumulants-of-traces}
k_l(\Tr \naturalMatrix^{p_1},\dots,\Tr \naturalMatrix^{p_l}).  
\end{equation}
As we will see, when $\naturalMatrix=\rho$ is a representation of the
unitary group $U(n)$,
one can treat $(\Tr \naturalMatrix^\rank)$ as a family of classical random
variables. Therefore the tensor cumulant in \eqref{eq:cumulants-of-traces} is in
fact a classical cumulant.

The \emph{microscopic} quantities describe the probabilistic behavior of the entries
of the random matrix $\naturalMatrix$; in particular we study the tensor cumulants
\begin{equation}
\label{eq:matrix-cumulants}
 \kappa_{p_1,\dots,p_l} := 
k_\rank(\naturalMatrix_{1 \gamma(1)}, \dots, \naturalMatrix_{\rank \gamma(\rank)}), 
\end{equation}
where $\rank=p_1+\cdots+p_l$ and $\gamma$ is the following permutation:
\begin{multline}
\label{eq:gamma}
\gamma:=(1,2,\dots,p_1)(p_1+1,p_1+2,\dots,p_1+p_2)\cdots \\
(p_1+\cdots+p_{l-1}+1,p_1+\cdots+p_{l-1}+2,\dots,p_1+\cdots+p_{l}).
\end{multline}
In the usual context of random matrix theory where the entries of the
matrix $\naturalMatrix$ commute, the quantities $\kappa_{p_1,\dots,p_l}$ and their
products are sufficient to describe the joint distribution of the entries of
$\naturalMatrix$. In order to deal with the case of random matrices with non-commutative
entries we need more information. 
It turns out that it is enough to
consider the family of quantities $\kappa_{(\V,\pi)}$ given by
\eqref{eq:kappa_concrete}. In particular, for an appropriate choice of
$(\V,\pi)$, they coincide with the quantities \eqref{eq:matrix-cumulants}:
\[ \kappa_{(\mathbf{1}_\rank,\gamma)} =  \kappa_{p_1,\dots,p_l} . \]

Higher order free probability theory studies the limits of the
quantities \eqref{eq:cumulants-of-traces} and \eqref{eq:matrix-cumulants} after
appropriate normalization, as the size $n$ of the matrix $\naturalMatrix$ tends
to infinity. We need to revisit the
proofs from the paper \cite{HigherOrderFreeness3} in order to ensure they
also apply in our non-commutative situation.

\subsection{Relation between macroscopic and microscopic quantities}
The following theorem gives the key relation between the macroscopic and
microscopic quantities describing a random matrix with non-commuting entries.

\begin{theorem}
\label{theo:moments-of-unitary}
If\/ $\naturalMatrix$ is an $n\times n$ unitarily invariant random matrix with non-commuting
entries then
\begin{equation}
\label{eq:moments-of-unitary}
 k_l(\Tr \naturalMatrix^{p_1},\dots, \Tr \naturalMatrix^{p_l})= 
\sum_{(\V,\pi)} \kappa_{(\V,\pi)}\  n^{\#(\gamma \pi^{-1})},
\end{equation}
where $\gamma$ is given by \eqref{eq:gamma} and the sum runs over partitioned
permutations $(\V,\pi)$ of the set $[\rank]$ such that
$\V\vee C(\gamma)=\mathbf{1}_{\rank}$, where $\rank=p_1+\cdots+p_l$.
\end{theorem}
\begin{proof}
This result follows from \cite[Equation (22)]{HigherOrderFreeness3}; however
for the sake of completeness and since in the aforementioned paper the cumulants
were defined in a seemingly different way via M\"obius inversion formula, we
present an alternative proof here.

There is a bijective correspondence between partitions $\widetilde{\W}$ of the set
$[l]$ and partitions $\W$ of the set $[\rank]$ such that $\W\geq
\gamma$; this bijection is given by replacing each element of the set $[l]$ by the block
corresponding to the appropriate cycle of $\gamma$.
We have
\begin{multline*}
\E_{\widetilde{\W}} \big[ \Tr \naturalMatrix^{p_1}, \dots ,\Tr  \naturalMatrix^{p_l} \big]\\
\begin{aligned}
 = &  \sum_{1\leq i_1,\dots,i_\rank\leq n} \E_\W \big[ \naturalMatrix_{i_1 i_{\gamma(1)}}, \dots,
\naturalMatrix_{i_\rank i_{\gamma(\rank)}}  \big]   \\ 
 = & \sum_{\V\leq \W}\ \sum_{1\leq i_1,\dots,i_\rank\leq n} k_\V \big[ \naturalMatrix_{i_1
i_{\gamma(1)}}, \dots, \naturalMatrix_{i_\rank i_{\gamma(\rank)}}  \big]  \\ 
= & \sum_{\pi\leq\V\leq \W}\ \sum_{1\leq i_1,\dots,i_\rank\leq n} 
[i_{\gamma(1)}=i_{\pi(1)}]\cdots [i_{\gamma(\rank)}=i_{\pi(\rank)}]\
\kappa_{(\V,\pi)}\\
= & \sum_{\pi\leq\V\leq\W} n^{\#(\gamma \pi^{-1})}\ \kappa_{(\V,\pi)},
\end{aligned}
\end{multline*}
where the third equality follows from 
Proposition \ref{prop:7}.
For a partition $\mathcal{U}\geq \gamma$ we define: 
\[ k_{\widetilde{\mathcal{U}}}\big[ \Tr \naturalMatrix^{p_1}, \dots ,\Tr  \naturalMatrix^{p_l} \big]:=  
\sum_{\substack{\pi\leq\V\leq\mathcal{U} \\ C(\gamma)\vee \V=\mathcal{U} }}
n^{\#(\gamma \pi^{-1})}\ \kappa_{(\V,\pi)}.
\]
According to this definition,  $k_{\widetilde{\mathcal{U}}}$ fulfills the moment-cumulant formula 
\eqref{eq:moment-cumulant-formula}. Since the cumulant is uniquely determined by this property, this finishes the proof.
\end{proof}

\subsection{Decay of the cumulants of entries}
\label{sec:star}

All considerations in this paper so far are exact and non-asymptotic.
In this section, we study
asymptotics of random matrices with non-commutative entries as
the size of the matrix tends to infinity.

For each $n\geq 1$, let $\naturalMatrix^{(n)}$ be an $n\times n$ unitarily invariant random
matrix with non-commuting entries. As before, we make the dependence in $n$
implicit and instead of $\naturalMatrix^{(n)}$ we simply write $\naturalMatrix$. This notation
applies to other quantities as well (for example $\kappa_{(\V,\pi)}$ depends
implicitly on $n$).

The following theorem is at the same time a definition of the quantities
$K_{(\V,\pi)}$ and $M_{p_1,\dots,p_l}$.

\begin{theorem}
\label{theo:moments-cumulants-formula-higher-order}
Assume that for every partitioned permutation $(\V,\pi)$ the limit
\begin{equation}
\label{eq:higher-order-cumulants-as-limits}
 K_{(\V,\pi)}:=\lim_{n\to\infty}  n^{|(\V,\pi)|}\ \kappa_{(\V,\pi)}  
\end{equation}
exists and is finite.
Then 
\begin{equation}
\label{eq:moments-cumulants-free}
 M_{p_1,\dots,p_l}:=\lim_{n\to\infty} n^{2(l-1)}\  k_l(\trn \naturalMatrix^{p_1},\dots, \trn \naturalMatrix^{p_l})=
\sum_{(\V,\pi)\leq (\mathbf{1}_\rank,\gamma)} K_{(\V,\pi)},
\end{equation}
where $\gamma$ is given by \eqref{eq:gamma}.
\end{theorem}
\begin{proof}
This is a special case of
\cite[Equation (35)]{HigherOrderFreeness3}. The only difficulty is that the paper
\cite{HigherOrderFreeness3} deals with random matrices 
with commuting entries. Therefore one has
to revisit the original proof in order to ensure that it applies to the
non-commutative situation. This is indeed the case thanks to Theorem
\ref{theo:moments-of-unitary}.

Since for other results in this Section we will need some basic ideas behind this proof, we will present here a short outline.
The proof from
\cite{HigherOrderFreeness3} relies on the fact that one can write Equation
\eqref{eq:moments-of-unitary}
in the form
\begin{multline}
\label{eq:famous-proof}
n^{2(l-1)}\ k_l(\trn \naturalMatrix^{p_1},\dots, \trn \naturalMatrix^{p_l})= \\ 
\sum_{\substack{(\V,\pi) \\ \V\vee C(\gamma)=\mathbf{1}_{\rank}}} 
\left( n^{|(\V,\pi)|} \kappa_{(\V,\pi)}\right) 
\frac{1}{n^{|(\mathbf{0},\gamma \pi^{-1})|+|(\V,\pi)|-|(\mathbf{1}_\rank,\gamma)|}}.
\end{multline}
The result follows from the fact that the following triangle inequality holds
true
\[{|(\mathbf{0},\gamma \pi^{-1})|+|(\V,\pi)|-|(\mathbf{1}_\rank,\gamma)|}\geq 0 \]
with the equality holding if and only if $(\V,\pi)\leq (\mathbf{1}_\rank,\gamma)$.
\end{proof}

If the above limits \eqref{eq:higher-order-cumulants-as-limits} and \eqref{eq:moments-cumulants-free} exist, it is convenient to think that the sequence
$\naturalMatrix^{(n)}$ of random matrices converges to some (abstract) limit object
$\naturalMatrix^{(\infty)}$. In the context of higher order free probability the quantities
$K_{(\V,\pi)}$ are called \emph{higher order free cumulants} of $\naturalMatrix^{(\infty)}$
and the quantities $M_{p_1,\dots,p_l}$ are called \emph{higher order moments}
of $\naturalMatrix^{(\infty)}$, cf \cite{HigherOrderFreeness3}.

The above theorem shows that the microscopic quantities describing a random
matrix uniquely determine their macroscopic counterparts. For our purposes it is
necessary to have also the opposite and to express the microscopic quantities
in terms of their macroscopic counterparts. However, in the
non-commutative case, this is not possible in general since the microscopic
quantities $\kappa_{(\V,\pi)}$ contain much more information than the
macroscopic quantities \eqref{eq:cumulants-of-traces}, as can be seen by 
a simple cardinality argument.
In order to have the description in the
opposite direction, one needs to assume that the entries of the matrices
under consideration asymptotically commute.

\subsection{Converse of the condition from Section \ref{sec:star}}
\label{subsec:converse_condition}

We say that a sequence $(\naturalMatrix)=(\naturalMatrix^{(n)})$ of unitarily invariant random matrices
(with non-commutative entries) has \emph{asymptotically vanishing commutators} 
up to degree $\len_0$ if 
\begin{multline}
\label{eq:vanishing-commutators}
k_{\V'} \big( \naturalMatrix_{1  \pi(1)}, \dots, \naturalMatrix_{i-1, \pi(i-1)},[\naturalMatrix_{i,
\pi(i)} ,\naturalMatrix_{i+1, \pi(i+1)}], 
\dots, \naturalMatrix_{\rank \pi(\len) }\big) \\
= o\left(\frac{1}{n^{|(\V,\pi)|}}\right)
\end{multline}
holds true for any partitioned permutation $(\V,\pi)$ of the set
$[\len]$, for any $\len\leq \len_0$ and any value of $i$ such that $i$ and
$i+1$ are connected by $\V$ and where $\V'$ should be understood as in Lemma
\ref{lem:cumulants-and-commutators}.

The following lemma and theorem provide the key induction step for the proof of
the main result of this paper, Theorem \ref{theo:main_theorem-RELOADED}.

\begin{lemma}
\label{lem:almost-commutes}
Let $(\naturalMatrix)$ be a sequence of random matrices which has asymptotically vanishing
commutators up to degree $\len_0$
and assume that the limits
\eqref{eq:higher-order-cumulants-as-limits} exist and are finite for all
partitioned permutations of the sets $[\len]$ for every
$\len<\len_0$. 

Then  
\[\lim_{n\to\infty}  n^{|(\V,\pi)|} \big( \kappa_{(\V,\pi)}-\kappa_{(\W,\sigma)}
\big)=0\]
whenever $(\V,\pi)$ and $(\W,\sigma)$ are conjugate partitioned  permutations of
the set $[\len]$ for $\len\leq \len_0$. 
\end{lemma}
\begin{proof}
From the multiplicativity of cumulants it follows that it is enough to prove
the lemma in the case when $\V=\W=\mathbf{1}$ is the partition consisting of
only one
block.  

It is possible to find a finite sequence of partitioned permutations $(\V,\pi)=(\textbf{1},\pi_0),\dots,(\textbf{1},\pi_l)=(\W,\sigma)$ which begins and ends with our partitioned permutations $(\V,\pi)$ and $(\W,\sigma)$ and such that each pair of neighbors in this sequence is conjugate by a transposition $(i,i+1)$ interchanging two neighboring elements. For this reason it is enough to show the lemma under additional assumption that
$\pi$ and $\sigma$ are conjugate by a transposition $(i,i+1)$ interchanging two
neighboring elements. But under the above assumptions this is a direct
application of Lemma \ref{lem:cumulants-and-commutators} and Equation
\eqref{eq:kappa_concrete}.
\end{proof}

\begin{theorem}
\label{theo:inductive-step}
Let $(\naturalMatrix)$ be a sequence of random matrices which has asymptotically
vanishing commutators up to degree $\len_0$. Assume that the limit
\eqref{eq:higher-order-cumulants-as-limits} exists for all partitioned
permutations $(\V,\pi)$ of the set $[\len]$ for all $\len<\len_0$.
Assume also that the limit \eqref{eq:moments-cumulants-free}
exists and is finite for all integers $p_1,\dots,p_l\geq 1$ such that
$p_1+\cdots+p_l\leq \len_0$.

Then the limit \eqref{eq:higher-order-cumulants-as-limits} exists for any
partitioned permutation $(\V,\pi)$ of the set $[\len]$ for $\len\leq
\len_0$.
Furthermore, $K_{(\V,\pi)}$ depends only on the conjugacy class of the
partitioned permutation $(\V,\pi)$.
\end{theorem}
\begin{proof}
We prove the claim by induction
with respect to $\len_0$.
Looking at Equation \eqref{eq:famous-proof}, one notices that
from the inductive hypothesis and multiplicativity of cumulants, 
every summand on the right hand side which corresponds to $\V$ consisting of
more than one block converges to a finite limit.
Thanks to Lemma \ref{lem:almost-commutes}, each summand for which $\V=\mathbf{1}_\rank$
consists of only one block can be rewritten in the form 
\[\bigg[ n^{|(\mathbf{1}_\rank,\gamma)|}
\kappa_{(\mathbf{1}_\rank,\gamma)} + o(1) \bigg] 
\frac{1}{n^{|(\mathbf{0},\gamma \pi^{-1})|+|(\V,\pi)|-|(\mathbf{1}_\rank,\gamma)|}},
\]
where $\gamma=\gamma_{p_1,\dots,p_l}$ with $p_1\geq \cdots\geq p_l$ given by
\eqref{eq:gamma} is a permutation conjugate to $\pi$ with cycles arranged in a
special way.

Thus we can view the collection of equations \eqref{eq:famous-proof} over
$p_1\geq \cdots\geq p_l$ such that $p_1+\cdots+p_l=\len_0$ as a system of
equations with the variables
$Q_{p_1,\dots,p_l}:=n^{|(\mathbf{1},\gamma_{p_1,\dots,p_l})|}
\kappa_{(\mathbf{1},\gamma_{p_1,\dots,p_l})}$, over $p_1\geq \cdots\geq p_l$
with $p_1+\cdots+p_l=\len_0$. In the limit $n\to\infty$ this system of equations
has a particularly simple form given by \eqref{eq:moments-cumulants-free} hence
it is upper-triangular (the relation $\leq$ on the set of partitioned permutations can be extended to a linear order). Therefore it is non-singular and by continuity it
remains non-singular for $n$ in some neighborhood of infinity. Solving this
system of equations thanks to Cramer formulas shows that the limit
\[ \lim_{n\to\infty} n^{|(\mathbf{1}_\rank,\gamma_{p_1,\dots,p_l})|}\ 
\kappa_{(\mathbf{1}_\rank,\gamma_{p_1,\dots,p_l})} \]
exists.

For an arbitrary partitioned permutation $(\V,\pi)$ the existence of the limit
follows from Lemma \ref{lem:almost-commutes} and multiplicativity of cumulants. 
Lemma \ref{lem:almost-commutes} also implies that the limit depends only on the
conjugacy class.
\end{proof}

\subsection{Stability of decay}

The decay speed of cumulants of random matrix moments seen in
\eqref{eq:moments-cumulants-free} is rather typical. The following lemma shows that
this kind of decay is stable under taking polynomial functions.

\begin{lemma}
\label{lem:decay-of-products}
For each $n\geq 1$, let $(I_\alpha^{(n)})_{\alpha\in A}$ be a collection
random variables. 
Assume that for any $l\geq 1$ and any choice of
$\alpha_1,\dots,\alpha_l$ the limit
\[ \lim_{n\to\infty} n^{2l-2}\ k(I_{\alpha_1},\dots,I_{\alpha_l}) \]
exists and is finite.
 
Then, the limit 
\[ \lim_{n\to\infty} n^{2l-2}\ k(P_1,\dots,P_l) \]
exists and is finite for any polynomials $P_1,\dots,P_l$ in variables
$(I_\alpha)$.
\end{lemma}

\begin{proof}
The assumption implies that
\begin{equation}
\label{eq:good-decay}
 \lim_{n\to\infty} n^{2 |\W|}\ k_{\W}(I_{\alpha_1},\dots,I_{\alpha_l})  
\end{equation}
exists and is finite for any choice of partition $\W$.

It is enough to show that the lemma holds true if each polynomial $P_i$ is a
monomial. Therefore it is enough to study the asymptotics of the expression
\begin{equation}
 \label{eq:ready-for-LS?}
 k\big( (I_{\alpha_1}\cdots I_{\alpha_{p_1}}),
(I_{\alpha_{p_1+1}} \cdots I_{\alpha_{p_1+p_2}}), \dots, (I_{\alpha_{p_1+\cdots+p_{l-1}+1}}\cdots
I_{\alpha_{p_1+\cdots+p_{l}}})\big). 
\end{equation}
We denote by 
\begin{multline*} 
\V:= \big\{ \{1,\dotsm,p_1\},\{p_1+1,\dots,p_1+p_2\},\dots,\\
\{p_1+\cdots+p_{l-1}+1,\dots,p_1+\cdots+p_l\} \big\}
 \end{multline*}
the corresponding partition.
From the formula of Leonov and Sirjaev \cite{LeonovSiraev} for classical cumulants it follows that
\eqref{eq:ready-for-LS?} is equal to 
\[ \sum_{\W : \V\vee\W=\mathbf{1} } k_{\W}(I_{\alpha_1},I_{\alpha_2},\dots). \]
Due to the combinatorial  inequality $|\V\vee\W|\leq |\V| + |\W|$, it
follows that for $\W$ which contribute to the sum, a bound $|\W| \geq l-1$ 
holds true, thus the assumption \eqref{eq:good-decay} finishes the proof.
\end{proof}

\subsection{Convergence in distribution in the sense of higher
order free probability}
\label{subsec:convergence-hops}

We say that a sequence $(\naturalMatrix)$ of random matrices \emph{converges in distribution
in the macroscopic sense of higher order free probability} if the limit
$M_{p_1,\dots,p_l}$ exists and is finite for any choice of integers
$p_1,\dots,p_l\geq 1$. Note that, in fact, this is a condition on the fluctuations of the (natural) spectral measures, provided that they exist.

We say that a sequence $(\naturalMatrix)$ of random matrices \emph{converges in distribution
in the microscopic sense of higher order free probability} if the limit
$K_{(\V,\pi)}$ exists and is finite for any choice of a partitioned permutation
$(\V,\pi)$.

With these notions, we can reformulate the results of this section. Theorem
\ref{theo:moments-cumulants-formula-higher-order} shows in particular that
the convergence in the microscopic sense implies the convergence in the
macroscopic sense while Theorem \ref{theo:inductive-step} shows that (under some
additional assumptions) the converse implication holds true as well. In
particular, in the case of classical random matrices (with commuting entries)
both notions are equivalent
and there is no need to make a distinction between them.

\pagebreak

\section{Representations and random matrices with non-commuting entries}
\label{sec:non-commuting-entries}

\subsection{Representation as a random matrix with non-commutative entries}
\label{subsec:representation-as-random-matrix}
We continue investigations from Section \ref{subsec:lie-representation-random-matrix-A} for the special case when $G=U(n)$ is the unitary group and $\g=\uu(n)$ its Lie algebra.

Under the notations from Section
\ref{subsec:matrix-structure-1},  we may view a representation
$\rho\in\uu(n)\gwia \otimes \End(V)$ as an $n\times n$ matrix with
entries in the non-commutative probability space $(\End(V),\trV)$ given
explicitly as a hermitian matrix
\begin{equation} 
\label{eq:reps-as-a-matrix}
\naturalMatrix(\rho) := 
\begin{bmatrix}
\rho(e_{11}) & \dots  & \rho(e_{1n}) \\
\vdots       & \ddots & \vdots       \\
\rho(e_{n1}) & \dots  & \rho(e_{nn}) 
\end{bmatrix} \in \M_n(\C) \otimes \End(V),
\end{equation}
where $e_{ij}\in \M_n(\C) = \uu(n) \otimes_\R \C$ are the matrix units.
We will use the notation $\naturalMatrix(\rho)$ in order to avoid ambiguities with other usages of the symbol $\rho$.

In the following, the normalized trace will enter in two distinct flavors: as the expected value $\trV$ in the non-commutative probability space \linebreak
$\big( \End(V), \trV \big)$, and as the normalized trace $\trn$ for matrices $\M_n(\C)$.

\subsection{Choice of the matrix structure on $\uu(n)\gwia$}
\label{subsec:matrix-structure}
Unlike in the case of the Lie algebra $\uu(n)$, there is no obvious
canonical choice of the matrix structure on its dual $\uu(n)\gwia$. In
Section \ref{subsec:matrix-structure-1}, this structure was
chosen based on a bilinear form $\langle A,B\rangle=\Tr A^T B$. One can argue
however, that a bilinear form $\langle A,B\rangle=\Tr A B$ would be equally
natural. This new way of choosing the matrix structure on $\uu(n)\gwia$
would have some advantages: for example the coadjoint action of $U(n)$ on it
corresponds to the usual adjoint action on $\M_n(\C)$ (without the somewhat
artificial complex conjugation). With respect to this new convention,
representation $\rho$ viewed as a matrix becomes
\begin{equation} 
\label{eq:reps-as-a-matrixA}
\begin{bmatrix}
\rho(e_{11}) & \dots  & \rho(e_{n1}) \\
\vdots       & \ddots & \vdots       \\
\rho(e_{1n}) & \dots  & \rho(e_{nn}) 
\end{bmatrix} \in \M_n(\C) \otimes \End(V).
\end{equation}
Matrices \eqref{eq:reps-as-a-matrix} and \eqref{eq:reps-as-a-matrixA} differ
just by transposition with respect to the first leg, also known as 
\emph{partial transpose}.
The advantage of the notation \eqref{eq:reps-as-a-matrix} is that it
coincides with the notation of \v{Z}elobenko \cite{Zelobenko73} which will be
useful later on in the calculation of the spectral measure. 

The calculation of the spectral measure of \eqref{eq:reps-as-a-matrixA} can be
done by the analogous methods to those of \v{Z}elobenko \cite{Zelobenko73};
the only difference is that instead of considering the tensor product with the
canonical representation, one should consider the tensor product with the
contragradient one; one obtains in this way a formula slightly different from the one from Proposition \ref{prop:estimate-zelobenko}.
This shows that, in fact, for the purposes of this article 
the two different definitions yield the same results.

\subsection{Representations of the unitary groups}
For any $\rank\geq 1$, we consider the following central element of the \emph{universal enveloping algebra} $\U(\uu(n))$:
\[ Z_{\rank}
=\sum_{1\leq i_1,\ldots ,i_\rank\leq n} 
 e_{i_1 i_2} e_{i_2 i_3} \cdots e_{i_\rank i_1}\in\U(\uu(n)).\]
We need the following result, due to {\v{Z}}elobenko \cite[Theorem 2,
p.~163]{Zelobenko73}. 
 
\begin{proposition}
\label{prop:estimate-zelobenko}
Let $\rho$ be an irreducible representation of\/ $\uu(n)$ corresponding
to the shifted highest weight $l=(l_1>\cdots>l_n)$. Then
\[ \rho(Z_\rank)=\sum_{1\leq i_1,\ldots ,i_\rank\leq n} 
 \rho(e_{i_1 i_2}) \rho(e_{i_2 i_3}) \cdots \rho(e_{i_\rank i_1})
 = \sum_{i=1}^n\gamma_i\ l_i^\rank,\]
where the number on the right-hand side should be understood as a multiple of
the identity operator and 
\[\gamma_i:= \prod_{j\neq i} \left( 1-\frac{1}{l_i-l_j} \right).\]
\end{proposition}

\subsection{Natural spectral measure of a representation}
\label{subsec:spectral-measure}
We define the \emph{natural spectral measure} 
\[\spectralNatural_\rho:=\spectralNatural_{\naturalMatrix(\rho)}\] 
of the representation $\rho$ of Lie group $U(n)$ or of Lie algebra $\uu(n)$ as the spectral measure (in the sense of Section \ref{subsec:random-matrices-noncommutative}) of the random matrix $\naturalMatrix(\rho)$ with non-commutative entries.

Note that we already defined the \emph{\naive} spectral measure $\spectralNaiveShifted_\rho$ of a representation in Section \ref{subsec:specral-measure-definition}. The purpose of this section
is to compare these two non-equivalent definitions.

\begin{proposition}
If $\rho_l$ is the irreducible representation corresponding to the
shifted highest weight $l$, then its natural spectral measure is given by
the deterministic probability measure
\[ \spectralNatural_{\rho_l}=\spectralNatural_l=\sum_i \frac{\gamma_i}{n} \
\delta_{l_i} ,\]
where $\gamma_i$ was defined in Proposition \ref{prop:estimate-zelobenko}.
\end{proposition}
\begin{proof}
Firstly, observe that for an irreducible representation
$\rho=\rho_l:\uu(n)\rightarrow\End(V)$
\begin{equation}
\label{eq:ecie-pecie-gdzie-jedziecie}
  \int_\R z^\rank \ d\spectralNatural_{\rho}(z) \stackrel{\mathcal{D}}{=}\trn
[\naturalMatrix(\rho)]^\rank =
\frac{1}{n}\ \rho(Z_\rank)\in\End(V). 
\end{equation}
This equality should be understood as follows: the left-hand side is a classical random variable which has the same distribution (i.e., the same moments) as the right-hand side which is a non-commutative random variable in the non-commutative probability space $\big( \End(V), \trV \big)$. On the other hand $Z_\rank\in\C[\Sy{\rank}]$ is a central
element thus the right-hand side is a multiple of the identity operator which can be identified with the scalar
\begin{equation}
\label{eq:skalar}
 \trV \trn [Y(\rho)]^r = \trace_{n \dimm V} [Y(\rho)]^r,   
\end{equation}
where on right-hand side we view $Y(\rho)$ as a matrix of size $n \dimm V$.

Thus the right-hand side of \eqref{eq:ecie-pecie-gdzie-jedziecie} has the same distribution as a constant random variable.
Since a probability measure concentrated in a single point is uniquely determined by its moments, 
it follows that the left-hand side is a deterministic random variable as well.
In other words, $\mu_\rho$ is a random probability measure, the sequence of moments of which is given, almost surely, by \eqref{eq:skalar}. It remains to study the probability measures which have this sequence of moments.

An example of such a measure is just the spectral measure of the hermitian matrix $Y(\rho)\in \M_{n \dimm V}(\C)$. Since this measure is compactly supported, it is uniquely determined by its moments. Thus $\mu_\rho$ is a deterministic probability measure, equal to the spectral measure of $Y(\rho)\in \M_{n \dimm V}(\C)$. It remains to study the latter measure.

From Proposition \ref{prop:estimate-zelobenko} and
\eqref{eq:ecie-pecie-gdzie-jedziecie}
it follows that the (possibly signed) probability measure 
\begin{equation}
\label{eq:spectral-measure-incorrect}
 \spectralNatural_{\rho_l}' := \left(1 - \sum_i \frac{\gamma_i}{n} \right) \delta_0 +
\sum_i \frac{\gamma_i}{n}\ \delta_{l_i} 
\end{equation}
fulfills
\begin{equation}
\label{eq:equal-moments}
 \int P(x) \ d\spectralNatural'_{\rho_l} = \trn P[\naturalMatrix(\rho)] 
\end{equation}
for every polynomial $P$. The above equality should be understood as follows: the left-hand side is a real number $x\in\R$,
 while the right-hand side is the appropriate multiple of identity $x \mathbf{1}_V$. 
Since both $\spectralNatural'_{\rho_l}$ given
by \eqref{eq:spectral-measure-incorrect} and the spectral measure of
$\rho_l$ are finitely supported, it follows immediately from
\eqref{eq:equal-moments} that they are equal. 

It remains to show that $\sum_i \gamma_i=n$, hence the first summand in
\eqref{eq:spectral-measure-incorrect} vanishes. This can be done by a careful
analysis of the proof of \v{Z}elobenko \cite{Zelobenko73}; we provide an
alternative proof below.

For $l=(l_1,\dots,l_n)$ and any integer $s$ we denote
$l+s=(l_1+s,\dots,l_n+s)$. Notice that
the
irreducible representation of the unitary group $\rho_{l+s}$ can be
explicitly written as $\rho_{l+s}(U) = (\det U)^{s} \rho_l(U)$ for
any $U\in U(n)$ hence the corresponding representation of the Lie algebra
fulfills $\rho_{l+s}(x) = s\Tr x\cdot 1+ \rho_l(x)$ for any $x\in
u(n)$. Therefore, if we view $\rho_l$ and $\rho_{l+s}$ as random
matrices with non-commutative entries then
\[ \rho_{l+s} = s 1 + \rho_l. \]
It follows that the spectral measure of $\rho_{l+s}$ is just the spectral
measure of $\rho_l$ shifted by $s$; on the other hand the measure
$\rho_{l+s}$ given by \eqref{eq:spectral-measure-incorrect} is equal to
the shifted measure $\rho_l$ only if $\sum_i \gamma_i=n$.
\end{proof}

In the case when the representation $\rho$ is not irreducible, its natural spectral
measure is a random probability measure on the real line which can be
interpreted as the natural spectral measure of a random irreducible representation
$\rho_l$ distributed according to \eqref{eq:probability-on-weights}.

It becomes clear that the \naive definition of the spectral measure
$\spectralNaiveShifted$ and the natural definition of the spectral measure
$\spectralNatural$ do not coincide. Nevertheless the following lemma shows that
they coincide asymptotically (under some mild technical assumptions).

\begin{lemma}
\label{lem:it-is-friday}
For each $\rank\geq 1$ there exist polynomials $P_\rank$ and $Q_\rank$ in $\rank+1$ variables such
that for any shifted highest weight $l$
\begin{align}
\label{eq:kuraki} 
\momentNumber_\rank (\spectralNaiveShifted_l) = &
P_\rank\big(n,  \momentNumber_1(\spectralNatural_l), \dots,
\momentNumber_{\rank}(\spectralNatural_l) \big), \\
\nonumber
\momentNumber_\rank (\spectralNatural_l) = &
Q_k\big(n,  \momentNumber_1(\spectralNaiveShifted_l), \dots,
\momentNumber_{\rank}(\spectralNaiveShifted_l) \big),
\end{align}
where 
\[\momentNumber_\rank(\mu)=\int_\R z^\rank \ d\mu(z)\]
denotes the $\rank$-th moment of a given measure $\mu$.

We define a degree on polynomials by assigning the degree $1$ to the variable $n$
and the degree $i$ to the variables
$\momentNumber_i(\spectralNatural_l)$ and $\momentNumber_i(\spectralNaiveShifted_l)$. 
Then the polynomials $P_\rank$ and $Q_\rank$ have degree $\rank$ and their leading terms are given by
\[ \momentNumber_\rank (\spectralNatural_l) + (\text{terms of degree $\rank$
which contain at least one factor $n$})\]
and
\[ \momentNumber_\rank (\spectralNaiveShifted_l) + (\text{terms of degree $\rank$
which contain at least one factor $n$})\]
respectively.
\end{lemma}
\begin{proof}
We denote
\[ L= 
\begin{bmatrix}
l_1  &  &  &  &  \\
  & l_2 &  &  &  \\
  &  & l_3 &  &  \\
  &  &  & \ddots &  \\
  &  &  &  & l_n
\end{bmatrix}, \qquad
J= 
\begin{bmatrix}
0 & -1 & -1 & \cdots  & -1     \\
  & 0  & -1 & \cdots  & -1     \\
  &    & 0  & \ddots  & \vdots \\
  &    &    & \ddots  & -1      \\
  &    &    &         & 0
\end{bmatrix}.\]

The (rescaled) moments of the spectral measure $\spectralNatural_l$ are given in
\cite{Zelobenko73}: 
\begin{equation}
\label{eq:zmiana-bazy}
\begin{split}
n\ \momentNumber_\rank(\spectralNatural_l)
= & \sum_{1\leq i,j\leq n} \big[ (L+J)^\rank \big]_{ij}\\
= & \sum_{\substack{\alpha_1,\dots,\alpha_q\geq 0, \\
\alpha_1+\cdots+\alpha_q+q-1=\rank }} \sum_{1\leq i,j\leq n} \big[ L^{\alpha_1} J
L^{\alpha_2} \cdots J L^{\alpha_{q}} \big]_{ij}  \\
= & \sum_{\substack{\alpha_1,\dots,\alpha_q\geq 0, \\
\alpha_1+\cdots+\alpha_q+q-1=\rank }} \sum_{1\leq i_1<\cdots<i_q\leq n}
(-1)^{q-1}\ l_{i_1}^{\alpha_1} \cdots l_{i_q}^{\alpha_q} \\
= & \sum_{\substack{\alpha_1\geq \cdots\geq \alpha_q\geq 0, \\
\alpha_1+\cdots+\alpha_q+q-1=\rank }} 
(-1)^{q-1}\ m_{(\alpha)}(l_1,\dots,l_n),
\end{split}
\end{equation}
where $m_{(\alpha)}$ denotes
the monomial symmetric polynomial. We allow here a small abuse of notation,
namely we allow some of the elements of $(\alpha)=(\alpha_1,\dots,\alpha_q)$ to be equal to zero; this
does not lead to problems since we treat $m_{(\alpha)}$ not as a symmetric
function but as a polynomial in a finite, fixed number of variables.

It is easy to check that by assigning to the expression $m_{(\alpha)}$ the
degree $(\alpha_1+1)+\cdots+(\alpha_q+1)$ one gets a filtration on the algebra generated by symmetric functions $(m_{(\gamma)})$; in other words any product $m_{(\alpha)} m_{(\beta)}$ can be written as a linear combination of monomial symmetric polynomials $m_{(\gamma)}$ such that $\degg m_{(\gamma)} \leq \degg m_{(\alpha)} + \degg m_{(\beta)}$. Furthermore,
\begin{multline}
\label{eq:power-sum-and-monomial}
p_{\alpha_1} p_{\alpha_2} \cdots = m_{(\alpha)} + \\
\left(\text{linear combination of $(m_{(\gamma)})$ such that $\degg m_{(\gamma)}< \degg m_{(\alpha)}$ }\right),
\end{multline}
where 
\[p_i(l_1,\dots,l_n)=l_1^i+\cdots+l_n^i\] 
for $i\geq 0$ are the power-sum symmetric polynomials.
The system of equations \eqref{eq:power-sum-and-monomial} is upper-triangular; it follows that for each $\alpha$ there exists some polynomial such that
\begin{multline*} 
m_{(\alpha)} = p_{\alpha_1} p_{\alpha_2} \cdots + \\
\left( \text{polynomial in $p_0,p_1,\dots$ of degree at most $\degg m_{(\alpha)}-1$}\right). 
\end{multline*}

The existence of the polynomial $Q_k$ follows now from \eqref{eq:zmiana-bazy} and the observation that
\begin{align*}
p_0(l_1,\dots,l_n)= & n\ \momentNumber_0(\spectralNaiveShifted)=n, \\
p_i(l_1,\dots,l_n)= & n\ \momentNumber_i(\spectralNaiveShifted)\qquad\qquad
\text{for $i\geq 1$}. 
\end{align*}
The passage
from quantities $(p_i)$ to $n$ and
$\big(\momentNumber_i(\spectralNaiveShifted)\big)$
corresponds to assigning to variable $n$ degree $1$ and to
$\momentNumber_i(\spectralNaiveShifted)$ degree $i$ which coincides with the
choice of degrees in the formulation of the lemma. The proof of the required
properties of the polynomial $(Q_k)$ is finished by the observation that the
right-hand side of \eqref{eq:zmiana-bazy} has degree $k+1$
and in order to have the minimal possible exponent standing at $n$ one should take
only the unique summand corresponding to $q=1$.

The family of equations \eqref{eq:kuraki} can be solved with
$\big(\momentNumber_k(\spectralNatural)\big)$ as unknowns which shows existence
of polynomials $(P_k)$ and their required properties.
\end{proof}

\subsection{Proof of the main result}

We come to the main result of this paper, Theorem
\ref{theo:main_theorem}, which we state in the following, more precise form.
We recall that various forms of convergence in the sense of higher order free probability have been defined in Section \ref{subsec:convergence-hops}.

\begin{theorem}
\label{theo:main_theorem-RELOADED}
For each $n$, let $\rho_n$ be a representation of the unitary group $U(n)$ and
assume that $\varepsilon_n=o\left(\frac{1}{n}\right)$. The
following conditions are equivalent:
\begin{enumerate}[label=(\alph*)]
 \item \label{item:macro}
the sequence 
$(\varepsilon_n \naturalMatrix(\rho_n))$ of natural random matrices with
non-com\-mu\-ta\-tive entries converges in distribution in the \emph{macroscopic sense of
higher order free probability},
 \item \label{item:micro}
the sequence 
$(\varepsilon_n \naturalMatrix(\rho_n))$ of natural random matrices with
non-com\-mu\-ta\-tive entries converges in distribution in the \emph{microscopic sense of
higher order free probability},
 \item \label{item:matrix}
the sequence $\big(\varepsilon_n
\naiveMatrix(\rho_n)\big)$ of \naive random matrices  converges in distribution in the \emph{sense of higher order free
probability}. 
\end{enumerate}
If the limits exist, they are equal (i.e., they describe the same
limiting object in the sense of higher order free probability) and, furthermore,
the limit $K_{(\V,\pi)}$ in \ref{item:micro} depends only on the conjugacy
class of $(\V,\pi)$. 
\end{theorem}
The fact that the limit $K_{(\V,\pi)}$ in \ref{item:micro} depends only on the conjugacy
class of $(\V,\pi)$ can be informally interpreted as asymptotic
commutativity of the entries of the matrices.

Condition \ref{item:macro} can be reformulated as a statement about fluctuations of the random probability measures $\mu_{\varepsilon_n \rho_n}$ (the natural spectral measures) while Condition \ref{item:matrix} can be reformulated as an analogous statement about the fluctuations of the \naive spectral measures $\spectralNaiveShifted_{\varepsilon_n \rho_n}$.

\begin{proof}
Assume that Condition \ref{item:macro} holds true. We will use induction over
$\len_0$ in order to prove \ref{item:micro}: assume that the limit
\eqref{eq:higher-order-cumulants-as-limits} exists for all
partitioned permutations $(\V,\pi)$ of the set $[\len]$ for all
$\len<\len_0$.
For $\naturalMatrix:=\varepsilon \naturalMatrix(\rho)$, we can write down explicitly the form of the
commutator on the left hand side of \eqref{eq:vanishing-commutators}:
\begin{multline*} 
[\naturalMatrix_{i,
\pi(i)} ,\naturalMatrix_{i+1, \pi(i+1)}]=[\varepsilon \rho(e_{i,\pi(i)}) , \varepsilon \rho(e_{i+1, \pi(i+1)})]
= \\ \varepsilon \times \Big( [\pi(i)=i+1]\ \varepsilon \rho(e_{i,\pi(i+1)})  -
[\pi(i+1)=i]\ \varepsilon \rho(e_{i+1,\pi(i)}) \Big)= \\
o\left(\frac{1}{n}\right) \times \Big(
[\pi(i)=i+1]\ \naturalMatrix_{i,\pi(i+1)}  -
[\pi(i+1)=i]\ \naturalMatrix_{i+1,\pi(i)}) \Big)
; 
\end{multline*}
thus the left-hand side of \eqref{eq:vanishing-commutators} is a linear combination of (at most) two expressions of the form
\begin{equation}
\label{eq:komutator-szacowanie}
 o\left(\frac{1}{n}\right) \times k_{\V'} \big( \naturalMatrix_{1  \sigma(1)}, 
\dots, \naturalMatrix_{\rank \sigma(\len-1) }\big) 
 \end{equation}
for some permutations $\sigma$ of the set $\{1,\dots,\len-1\}$. 
Note that a priori such a statement would require renumbering of the rows and columns of the matrix $\naturalMatrix$. 
But due to the unitary invariance of $\naturalMatrix$, renumbering does not change the joint distribution of the entries of the matrix 
$\naturalMatrix$.
The inductive assumption about the limits \eqref{eq:higher-order-cumulants-as-limits} gives a bound for the 
quantity \eqref{eq:komutator-szacowanie}.
One can check that the permutation $\sigma$ has the same number of cycles as $\pi$. Thus, $ |(\V',\sigma)|=|(\V,\pi)|-1$.
This implies that
the sequence $(\naturalMatrix)$ has asymptotically vanishing commutators up to
the order $\len_0$. Therefore, Theorem  \ref{theo:inductive-step} can be applied
and the limit \eqref{eq:higher-order-cumulants-as-limits} exists for all
partitioned permutations $(\V,\pi)$ of the set $[\len]$ for all
$\len\leq \len_0$, thus we finished the proof of the induction step. In this way we proved Condition \ref{item:micro}. 

The opposite implication \ref{item:micro}$\implies$\ref{item:macro} follows directly from Theorem
\ref{theo:moments-cumulants-formula-higher-order}.

In order to show the implication \ref{item:matrix}$\implies$\ref{item:macro},  we
need to show that the cumulant
\begin{equation}
 \label{eq:hungry}
  k_l\big(\varepsilon^{\rank_1}
\momentNumber_{\rank_1}(\spectralNatural_{\rho}),\dots,
\varepsilon^{\rank_l} \momentNumber_{\rank_l}(\spectralNatural_{\rho})\big)
\end{equation}
converges sufficiently quickly  to zero. In order to do this, we use Lemma
\ref{lem:it-is-friday} and express \eqref{eq:hungry} in terms of the cumulants
of polynomials in $\big(\momentNumber_i(\spectralNaiveShifted_\rho)\big)$. Lemma
\ref{lem:decay-of-products} finishes the proof. 

The opposite implication \ref{item:macro}$\implies$\ref{item:matrix} can be
proved in an analogous way.

Finally, the equality of the limiting objects is obvious. This completes the proof.
\end{proof}

\section{Applications to asymptotic representation theory}
\label{sec:applications}

In this section we provide some concrete applications of the main result
\ref{theo:main_theorem-RELOADED}.
In particular, we elaborate on Section
\ref{subsec:concrete-applications}
and supply proofs of the results announced in the introduction.

\subsection{Gaussianity of fluctuations for Kronecker tensor product}

\begin{proof}[Proof of Corollary \ref{corollary:Kronecker-RM}]
We will show that for every $i\in\{1,2\}$ the
sequence $(\rho^{(i)}_n)$ of representations fulfills condition \ref{item:matrix} of Theorem \ref{theo:main_theorem-RELOADED}. 
Indeed, the convergence of the sequence $\left( \varepsilon_n \naiveMatrix(\rho_n) \right)$ in the macroscopic sense of higher order free probability is, by definition, equivalent to existence of the limits \eqref{eq:moments-cumulants-free}.
For $l=1$, this limit is just
\[ M_{p} = \lim_{n\to\infty} \E \trn \left[\varepsilon_n \naiveMatrix(\rho_n^{(i)})\right]^p.\]
The existence of this limit is equivalent to the assumption that the \naive 
spectral measures $\left(\spectralNaiveShifted_{\varepsilon_n \rho^{(i)}_n}\right)_{n=1,2,\dots}$  converge in moments to some limit.
For $l\geq 2$ the cumulants
\[ k_l\left(\trn [\naiveMatrix(\rho_n^{(i)})]^{p_1},\dots, \trn [\naiveMatrix(\rho_n^{(i)})]^{p_l}\right)=0 \]
vanish, because each random variable $\trn [\naiveMatrix(\rho_n^{(i)})]^{p_k}$ is, in fact, a constant one; thus the limit $M_{p_1,\dots,p_l}=0$ exists trivially.

Since the
sequence $(\rho^{(i)}_n)$ of representations fulfills condition \ref{item:matrix} of Theorem \ref{theo:main_theorem-RELOADED},
it follows that it also fulfills condition \ref{item:micro}:  the sequence $\left(\varepsilon_n
\naturalMatrix(\rho^{(i)}_n)\right)$ converges in the microscopic sense of higher order
probability theory and that this microscopic limit is the same as for the
sequence of random matrices $\left(\varepsilon_n \naiveMatrix(\rho^{(i)}_n)\right)$.

We denote by $\rho^{(3)}_n:=\rho^{(1)}_n \otimes \rho^{(2)}_n$ the Kronecker tensor product of representations.
For Lie algebras representations $\rho^{(i)}:\uu(n)\rightarrow\End
V^{(i)}$, $i\in\{1,2\}$, it follows that  
\[\rho^{(3)}_n(x) = \rho^{(1)}_n(x)\otimes 1 + 1\otimes \rho^{(2)}_n(x)
\in \End(V^{(1)} \otimes V^{(2)}) \]
for any $x\in\uu(n)$, hence
\begin{equation}
\label{eq:aaaa}
\varepsilon_n \naturalMatrix(\rho^{(3)}) = \varepsilon_n \naturalMatrix(\rho^{(1)})\otimes 1 + 1\otimes \varepsilon_n \naturalMatrix(\rho^{(2)})
\in \M_n(\C)\otimes \End(V^{(1)} \otimes V^{(2)}). 
\end{equation}

On the other hand, if $\naiveMatrix^{(i)}\in \M_n(\C) \otimes \mathcal{L}^{\infty-}(\Omega^{(i)})$, 
$i\in\{1,2\}$ are random matrices, the sum of their independent
copies can be realized on the product probability space $\Omega^{(1)}\times
\Omega^{(2)}$ as 
\begin{equation}
\label{eq:aaaa2} 
 \widetilde{\naiveMatrix}^{(3)}:= \naiveMatrix^{(1)}\otimes 1 + 1\otimes \naiveMatrix^{(2)}
\in \M_n(\C) \otimes \mathcal{L}^{\infty-}(\Omega^{(1)}\times \Omega^{(2)}). 
\end{equation}

Each of the expressions \eqref{eq:aaaa} and \eqref{eq:aaaa2} is a sum of two
(non-commutative) random vectors in $\M_n(\C)$  which have tensor
independent coordinates; each of these summand converges in the microscopic sense of higher order free probability and the limits of the first (respectively, second) summands are equal. By additivity of cumulants, 
it follows immediately that also $\varepsilon_n
\naturalMatrix(\rho^{(3)}_n)$ and $\widetilde{\naiveMatrix}^{(3)}$ converge in the microscopic sense of higher
order free probability theory and that the limits are equal. 

We apply Theorem \ref{theo:main_theorem-RELOADED} again and show that
$\varepsilon_n \naturalMatrix(\rho^{(3)}_n)$, $\naiveMatrix^{(3)}$ and $\widetilde{\naiveMatrix}^{(3)}$ converge in the
macroscopic sense of higher order free probability theory and that their limits
are equal.

For a sequence $(\naiveMatrix_n)$ of random matrices, the convergence in the macroscopic
sense of higher order free probability theory is equivalent to existence of the
limits \eqref{eq:moments-cumulants-free} and implies in particular that that the limits
\begin{align}
\label{eq:mean}
 M_l          = & \lim_{n\to\infty} \E \trn \naiveMatrix^l, \\
\label{eq:kowariancja}
 M_{l_1,l_2} = &\lim_{n\to\infty} \Cov\left( n\trn \naiveMatrix^{l_1},n \trn \naiveMatrix^{l_2}
\right), \\
\nonumber
 &  \lim_{n\to\infty} 
k_i\left( n\trn \naiveMatrix^{l_1},\dots, n \trn \naiveMatrix^{l_i} \right) =0 \qquad \text{for $i\geq
3$} 
\end{align}
exist. In classical probability theory, vanishing of the cumulants (other than the mean value and variance) characterizes the Gaussian distribution;
it shows that the spectral measure of $\naiveMatrix_n$ has
asymptotically Gaussian fluctuations with covariance decay $\frac{1}{n^2}$ which finishes the proof that the spectral
measures (both the \naive and the natural ones) have the same Gaussian
fluctuations as random matrices $\widetilde{\naiveMatrix}^{(3)}$.
\end{proof}

\subsection{Almost surely convergence}

\begin{proof}[Proof of Corollary \ref{corollary:Kronecker-free}]
For a sequence $(\naiveMatrix)$ of random matrices which converges in the macroscopic
sense of higher order free probability,
Equation \eqref{eq:kowariancja} shows that for every value of $l\geq 1$
\[ \Var \trn \naiveMatrix^l = O\left( \frac{1}{n^2} \right) \]
so Chebyshev's inequality together with Borel-Cantelli lemma show that 
$ \trn \naiveMatrix^l $ converges to \eqref{eq:mean} almost surely.

Since the spectral measure of the sum of independent random matrices
concentrates around Voiculescu's free convolution of their spectral measures
\cite{Voiculescu1991}, the results presented in the above proof of Corollary
\ref{corollary:Kronecker-RM} finish the proof.
\end{proof}

\subsection{Restriction to the subgroup}

\begin{proof}[Proof of Corollary \ref{coro:restriction}]
Just like in the proof of  Corollary \ref{corollary:Kronecker-RM} above,
we show that  
$\varepsilon_n \naiveMatrix(\rho_n)$ converges in the macroscopic sense of higher order free probability theory thus
$\varepsilon_n \naturalMatrix(\rho_n)$ converges in the microscopic sense of higher order free probability theory.

It follows immediately that
\begin{multline} 
\label{eq:restrictionA}
K_{(\V,\pi)}(\rho'):= \lim_{n\to\infty} m_n^{|(\V,\pi)|} \kappa_{(\V,\pi)} = \\
\alpha^{|(\V,\pi)|} \lim_{n\to\infty} n^{|(\V,\pi)|} \kappa_{(\V,\pi)}=
\alpha^{|(\V,\pi)|}  K_{(\V,\pi)}(\rho);
\end{multline}
in particular $\varepsilon_n \naturalMatrix(\rho'_n)$ converges in the microscopic sense of higher order free probability theory.

If we define $\naiveMatrix_n'$ as the $m_n\times m_n$ upper-left corner of the random matrix $\naiveMatrix_n$, an analogous calculation shows that
\begin{equation}
\label{eq:restrictionB}
 K_{(\V,\pi)}(\naiveMatrix')=
\alpha^{|(\V,\pi)|}  K_{(\V,\pi)}(\naiveMatrix);  
\end{equation}
as the right-hand sides of \eqref{eq:restrictionA} and \eqref{eq:restrictionB} are equal, so must be their left-hand sides. 

In an analogous way as in the proof of Corollary \ref{corollary:Kronecker-RM} it follows that the rescaled \naive spectral measure $\spectralNaiveShifted_{\varepsilon_n\rho'_n}$ of the restricted representation
and the spectral measure of the $m_n\times
m_n$ upper-left corner of the random matrix $\naiveMatrix_n$ have asymptotically 
the same Gaussian fluctuations with covariance decay $\frac{1}{n^2}$.

In an analogous way as in the proof of Corollary \ref{corollary:Kronecker-free}
one can show that the (rescaled) \naive spectral measures of $\rho'$ and the spectral measures of $X'$ converge almost surely to the same limit. On the other hand it is well-known that the spectral measures of $X'$ converge to the free compression of the measure $\mu$, which finishes the proof.
\end{proof}

\section{Acknowledgments} \label{sec:Acknowl}

Both authors thank J.~Mingo and R.~Speicher for many useful discussions about 
higher order free probability theory.

The research of P.\'S.~was supported by the research grant of Polish Ministry of
Science and Higher Education number \mbox{N N201
364436} for the years 2009--2012.

The research of B.C.\ was partly supported by the NSERC grant \linebreak RGPIN/341303-2007,
  the ANR grants Galoisint and Granma, and 
   a Marie Curie Transfer of Knowledge Fellowship of the European Community's Sixth Framework
Programme under contract MTKD-CT-2004-013389.

\bibliographystyle{alpha}

\bibliography{biblio-minimal}

\end{document}